\documentclass[tbtags,a4paper,12pt]{amsart}

\usepackage[in]{fullpage}

\usepackage{amsmath,amssymb,amsthm,amsfonts,amscd}

\newcommand{\codim}{\mathop{\rm codim}\nolimits}
\newcommand{\pt}{\mathrm{pt}}
\newcommand{\cone}{\mathop{\mathrm{cone}}\nolimits}
\newcommand{\Lk}{\mathop{\mathrm{link}}\nolimits}
\newcommand{\Hom}{\mathop{\mathrm{Hom}}\nolimits}

\newcommand{\T}{\mathcal{T}}
\newcommand{\B}{\mathcal{B}}
\newcommand{\Q}{\mathbb{Q}}
\newcommand{\CQ}{\mathcal{Q}}

\newcommand{\Z}{\mathbb{Z}}
\newcommand{\R}{\mathbb{R}}
\newcommand{\hc}{h}

\newcommand{\St}{\mathop{\mathrm{star}}\nolimits}

\newtheorem{theorem}{Theorem}[section]
\newtheorem{propos}[theorem]{Proposition}
\newtheorem{cor}[theorem]{Corollary}

\theoremstyle{definition}

\newtheorem{defin}[theorem]{Definition}
\newtheorem{remark}[theorem]{Remark}

\numberwithin{equation}{section}

\author{Alexander A. Gaifullin}

\thanks{The work is partially supported by Russian Foundation for Basic Research 
(grants 08-01-00541 and 08-01-91855) and The Grant of The President of Russian Federation  (grant MK-4220.2009.1).}

\title[Combinatorial formulae for Pontryagin classes]{Configuration spaces, bistellar moves, and combinatorial formulae for the first Pontryagin class}

\date{}

\address{Moscow State University}

\address{Institute for Information Transmission Problems}

\begin{document}

\begin{abstract}
The paper is devoted to the problem of finding explicit combinatorial formulae for the Pontryagin classes. We discuss two formulae, the classical Gabrielov--Gelfand--Losik formula based on investigation of configuration spaces and the local combinatorial formula obtained by the author in 2004. The latter formula is based on the notion of a universal local formula introduced by the author and on the usage of bistellar moves. 
We give a brief sketch for the first formula and a rather detailed exposition for the second one. For the second formula, we also succeed to simplify it by providing a new simpler algorithm for decomposing a cycle in the graph of bistellar moves of two-dimensional combinatorial spheres into a linear combination of elementary cycles.   
\end{abstract}

\maketitle

\section{Introduction}

The problem of finding combinatorial formulae for the Pontryagin classes of triangulated manifolds goes back to a remarkable work~\cite{GGL75} by A.\,M.~Gabrielov, I.\,M.~Gelfand, and M.\,V.~Losik, where they constructed the first explicit combinatorial formula for the first rational Pontryagin class. Later different combinatorial formulae were obtained in~\cite{GGL76}--\cite{Gai08}. Let us also mention that N.~Levitt and C.~Rourke~\cite{LR78} proved the existence of local combinatorial formulae for all polynomials in rational Pontryagin classes without constructing explicit formulae. The simplest of all known combinatorial formulae for the first rational Pontryagin class was obtained by the author in 2004~\cite{Gai04}. The survey~\cite{Gai05} is devoted to the comparison of different combinatorial formulae for the Pontryagin classes. 

Before passing to a detailed discussion of combinatorial formulae for the Pontryagin classes 
we shall briefly discuss a simpler problem of combinatorial computation of the Stiefel--Whitney classes.
The following assertion was conjectured by E.~Stiefel~\cite{Sti36} and was proved by  H.~Whitney~\cite{Whi40}.

\begin{theorem}\label{theorem_SW}
Suppose $K$ is a closed $n$-dimensional combinatorial manifold, and $K'$ is its first barycentric subdivision. Denote by $W_k$ the sum modulo~$2$ of all $k$-dimensional simplices of~$K'$. Then the simplicial chain~$W_k$ is a cycle with coefficients in~$\Z_2$ and represents the Poincar\'e dual of the Stiefel--Whitney class~$w_{n-k}(K)$.
\end{theorem}

H.~Whitney had not published a detailed proof of this theorem. An accurate proof of it was proved only in 1972 by S.~Halperin and D.~Toledo~\cite{HT72}. Their proof as well as the original proof by H.~Whitney is based on an explicit construction of tangent vector fields $F_1,\ldots,F_n$ on~$K$ such that the fields $F_1,\ldots, F_p$ are linearly independent off the $(n-p)$-skeleton of~$K'$ and the index of the vector field~$F_p$ modulo $F_1,\ldots,F_{p-1}$ at the barycentre of each $(n-p)$-dimensional simplex of~$K'$ is equal to~$\pm 1$.
Another proof of Thorem~\ref{theorem_SW} based on quite different ideas was obtained by J.~Cheeger~\cite{Che70}.

In this paper we consider two combinatorial formulae for the first rational Pontryagin class, the classical Gabrielov--Gelfand-Losik formula~\cite{GGL75} and the formula obtained by the author in~\cite{Gai04}. Indeed, these two formulae are at the moment the only known formulae for the Pontryagin classes that can be used for real computation. All known formulae for higher Pontryagin classes (see~\cite{Gab78}, \cite{GMP92},\cite{Che83},~\cite{Gai08}) cannot be used for real computation. The formula due to A.\,M.~Gabrielov~\cite{Gab78} can be applied to a very small class of triangulations of manifolds only. The formula due to I.\,M.~Gelfand and R.~MacPherson is not purely combinatorial~\cite{GMP92}, since it computes the Pontryagin classes of a manifold from a given triangulation with a given \textit{smoothing} rather than from a given triangulation only. Though the formula due to the author~\cite{Gai08} is purely combinatorial, it is extremely complicated, hence, it cannot be used for real computation even in simplest cases. We shall mention apart an analitic approach developing the results of Atiyah--Patodi--Singer~\cite{APS73} which allowed J.~Cheeger~\cite{Che83} to obtain formulae for the Pontryagin classes of triangulated manifolds in terms of the spectra of the Laplace operators in spaces of $L_2$-forms on incomplete Riemannain manifolds with locally flat metrics.        
These formulae can be applied to any combinatorial manifold. Nevertheless, the spectra of the Laplace operators also lack an explicit combinatorial description. Hence Cheeger's formulae should be regarded as important relations between topological and analitic objects rather than as formulae for combinatorial computation of the Pontryagin classes. 

Let us now describe main ideas of the Gabrielov--Gelfand--Losik formula~\cite{GGL75} and the author's formula~\cite{Gai04}. For the Gabrielov--Gelfand--Losik formula there are two almost equivalent approaches. The original approach in~\cite{GGL75} is based on endowing of a triangulated manifold with locally flat connections. Another approach due to R.~MacPherson~\cite{MP77} is based on the construction of homology Gaussian mapping for a combinatorial manifold. We shall describe the ideas of MacPherson's approach, since it is more geometrical.

First suppose that $M^m\subset\R^N$ is an oriented smooth closed manifold embedded into a Euclidean space of large dimension. Consider the Gaussian mapping 
$g:M\to G_m(\R^N)$ taking every point~$x\in M$ to the linear subspace of~$\R^N$ parallel to the tangent space to~$M$ at~$x$. Here $G_m(\R^N)$ is the Grassmannian manifolds of
$m$-dimensional subspaces of~$\R^N$. For any homogeneous polynomial  $F\in \R[p_1,p_2,\ldots,p_{[m/2]}]$ of degree $n=4k$, there is a unique $O(N)$-invariant interior $n$-form~$P_F$ on~$G_m(\R^N)$ representing the class $F(p(\gamma))=F\bigl(p_1(\gamma),p_2(\gamma),\ldots,p_{[m/2]}(\gamma)\bigr)$ in the de Rham cohomology,
where $\gamma$ is the tautological vector bundle over~$G_m(\R^N)$. (Here and always in the sequel under a \textit{homogeneous} polynomial of degree~$n$ in Pontryagin classes we mean a polynomial all whose monomials have degree~$n$, where the degree of every variable~$p_i$ is supposed to be~$4i$.) 
Then the form~$g^*P_F$ represents the class~$F(p(M))$ in the de Rham cohomology.

Obviously, for combinatorial manifolds, there is no hope to construct a natural \textit{continuous} mapping to~$G_m(\R^N)$. Indeed, if such mapping had existed, we could have defined \textit{integral} Pontryagin classes of a combinatorial manifold to be the pullbacks of the Pontryagin classes of the tautological bundle~$\gamma$. This is certainly impossible, since integral Pontryagin classes are not invariant under piecewise linear homeomorphisms. However, for combinatorial manifolds, it is possible to construct a \textit{homology} Gaussian mapping up to dimension~$4$, which yields the required combinatorial formula. It is interesting that the formula obtained is essentially non-local. The reason is that the definition of the homology Gaussian mapping is non-local, since it depends on the preliminarily chosen additional combinatorial structure on the combinatorial manifold. A local formula is obtained by certain special averaging of this additional structure, which makes the formula much more complicated (see~\cite{GGL76},~\cite{MP77}). Notice that the Gabrielov--Gelfand--Losik formula cannot be applied for an arbitrary combinatorial manifold. It can be applied only for a smaller class of  so-called \textit{Brouwer manifolds} (for definition, see section~\ref{section_triangman}).

The author's formula~\cite{Gai04} is based on different ideas. Notice that in smooth case the form~$g^*P_F$ gives a \textit{local} formula for the polynomial~$F$ in Pontryagin classes. This means that the form~$g^*P_F$ at every point depends on the metrics on~$M$ in the neighbourhood of this point only. Hence it seems natural to hope that a cycle representing the Poincar\'e dual of a given polynomial in Pontryagin classes can be given by a \textit{universal local formula}
\begin{equation*}
f_{\sharp}(K)=\sum_{\sigma\in
K,\,\codim\sigma=n}f(\Lk\sigma)\sigma,
\end{equation*}
where $f$ is a chosen rational-valued function on the set of isomorphism classes of oriented $(n-1)$-dimensional combinatorial spheres independent of the combinatorial manifold~$K$. In~\cite{Gai04}, the author suggested the following approach to constructing an explicit combinatorial formula for the first Pontryagin class. For $n=4$, it is possible to use the theory of bistellar moves to describe explicitly \textit{all} functions~$f$ such that the chain~$f_{\sharp}(K)$ is a cycle for every combinatorial manifold~$K$. Any such function~$f$ gives a formula for the first Pontryagin class multiplied by some rational number. The obtained explicit combinatorial formula is automatically local and can be applied to any combinatorial manifold.

This paper is organized in the following way. Section~\ref{section_triangman} contains main definitions and notation and necessary background information on triangulated manifolds. In section~\ref{section_GGL} we discuss configuration spaces, which form the most important ingredient of the Gabrielov--Gelfand--Losik formula, and give a very brief sketch of this formula. Sections~\ref{section_locform}--\ref{section_explicit} are devoted to the author's formula obtained in~\cite{Gai04}. In section~\ref{section_locform}, we construct a differential graded algebra~$\T_*$ of oriented combinatorial spheres, introduce the notion of a universal local formula, and formulate theorems of uniqueness and existence of local formulae for polynomials in rational Pontryagin classes. In section~\ref{section_bistellar} we give necessary information on bistellar moves and construct a graph~$\Gamma_n$ whose vertices are isomorphism classes of oriented $n$-dimensional combinatorial spheres and whose edges are equivalence classes of bistellar moves. We also give a new simpler algorithm for decomposing a cycle of the graph~$\Gamma_2$ into a linear combination of elementary cycles (Proposition~\ref{propos_cycly}). This algorithm essentially simplifies the explicit combinatorial formula for the first Pontryagin class obtained in~\cite{Gai04}. The construction of this explicit formula is given in section~\ref{section_explicit}.

\section{Triangulated manifolds}

\label{section_triangman} This section contains a survey of necessary definitions and results on simplicial complexes and triangulated manifolds.

An \textit{abstract simplicial complex} on a vertex set~$V$ is a non-empty set~$K$ of subsets of~$V$ such that на $\tau\in K$ whenever $\sigma\in K$ and $\tau\subset\sigma$. In particular, the empty set~$\emptyset$ always belongs to~$K$. Elements of~$K$ are called \textit{simplices} of~$K$. The dimension of a simplex $\sigma\in K$ is, by definition, the cardinality of the set~$\sigma$ minus~$1$. We shall always assume that the set~$V$ is finite and all one-element subsets of~$V$ belong to~$K$. A \textit{subcomplex} of~$K$ is a subset~$L\subset K$ that is a simplicial complex on some vertex set $W\subset V$. A \textit{full subcomplex} of~$K$ spanned by a subset $W\subset V$ is a subcomplex consisting of all simplices $\sigma\in K$ such that $\sigma\subset W$. An \textit{isomorphism} of simplicial complexes $K_1$ and $K_2$ on vertex sets $V_1$ and $V_2$ respectively is a bijection  $f:V_1\to V_2$ such that $f(\sigma)\in K_2$ if and only if $\sigma\in
K_1$.  Isomorphism of simplicial complexes is denoted by the symbol~$\cong$. A \textit{join} of simplicial complexes $K_1$ and $K_2$ on vertex sets $V_1$ and $V_2$ respectively is a simplicial complex~$K_1*K_2$ on the vertex set $V_1\sqcup V_2$ consisting of all simplices $\sigma_1\cup\sigma_2$ such that $\sigma_1\in K_1$ and $\sigma_2\in K_2$. The \textit{cone} over a simplicial complex~$K$ is the simplicial complex
$\cone K=\pt*K$, where $\pt$ is a one-point simplicial complex. The \textit{star} of a simplex $\sigma\in K$ is the subcomplex  $\St\sigma\subset K$ consisting of all simplices~$\tau$ such that $\sigma\cup\tau\in K$.
The \textit{link} of a simplex $\sigma\in K$ is the subcomplex $\Lk\sigma\subset K$ consisting of all simplices~$\tau$ such that $\sigma\cup\tau\in K$ and $\sigma\cap\tau=\emptyset$. Obviously,  $\St\sigma\cong\sigma*\Lk\sigma$, where the simplicial complex consisting of the simplex~$\sigma$ and all its faces is also denoted by~$\sigma$. The set of vertices of a simplicial complex~$K$ will be denoted by~$V(K)$ and will be identified with the set of all zero-dimensional simplices of~$K$.

Let us embed the set $V(K)$ as an affinely independent subset into some space~$\R^N$. To each abstract simplex $\sigma\in K$ we assign the convex simplex of dimension $\dim\sigma$ with vertices at the corresponding points of~$\R^N$. The union of all such simplices is a polyhedron called a \textit{geometric realisation} of~$K$ and denoted by~$|K|$. A geometric realisation does not depend on the number~$N$ and the embedding~$V(K)\hookrightarrow\R^N$ up to a piecewise linear homeomorphism.

A simplicial complex~$K$ is called a \textit{simplicial manifold} if $|K|$ is a topological manifold.  (Under a manifold we always mean a closed manifold, that is, a compact manifold without boundary.) A simplicial complex is called an \textit{$n$-dimensional combinatorial sphere} if its geometric realisation is piecewise linearly homeomorphic to the boundary of an $(n+1)$-dimensional simplex. A simplicial complex is called an \textit{$n$-dimensional combinatorial manifold} if the link of every its vertex is an $(n-1)$-dimensional combinatorial sphere.  

It is well known that the link of every simplex of a combinatorial sphere is a  combinatorial sphere itself. (see, for example,~\cite{RS74}). Hence any combinatorial sphere is a combinatorial manifold and the link of every simplex of a combinatorial manifold is a combinatorial sphere. Any combinatorial manifold is a simplicial manifold. The inverse assertion is true for~$n\le 4$ and false for $n\ge 5$. The simplest example of a non-combinatorial triangulation of a $5$-dimensional sphere is a double suspension over an arbitrary triangulation of a three-dimensional homology sphere.

A \textit{flattening} of a simplicial manifold~$K$ at its simplex~$\sigma$ of codimension~$k$ is an embedding 
$\varphi:|\cone\Lk\sigma|\hookrightarrow\R^k$ such that $\varphi$ takes the cone vertex to the origin and is linear on every simplex of~$\cone\Lk\sigma$. A simplicial manifold is called a \textit{Brouwer manifold} if it admits such embedding for every its non-empty simplex~$\sigma$. Any Brouwer manifold is a combinatorial manifold. The inverse assertion is true for $n\le 3$ and false for $n\ge 4$, since, for every $n\ge 4$, there is an $(n-1)$-dimensional combinatorial sphere $L$ such that the complex $\cone L$ cannot be embedded into~$\R^n$ so that the restriction of the embedding to each simplex is linear. However, J.\,G.\,C.~Whitehead~\cite{Whi41} proved that any combinatorial manifold has a stellar subdivision that is a Brouwer manifold.

For simplicity we shall always work with oriented maifolds. All results can be easily extended to the case of non-orientable manifolds by replacing usual simplicial chains with so-called \textit{co-oriented simplicial chains} (see~\cite{GGL75}). Unless otherwise is stated, under an isomorphism of oriented combinatorial manifolds we shall alway mean an orientation preserving isomorphism and we shall always use the symbol~$\cong$ for orientation preserving isomorphisms only. For any oriented combinatorial manifold~$K$ we denote by $-K$ the combinatorial manifold~$K$ with the orientation reversed.

\section{Configuration spaces and the Gabrielov--Gelfand--Losik formula}\label{section_GGL}

The main ingredient of the Gabrielov--Gelfand--Losik formula is so-called \textit{configuration spaces}.

Let $L$ be an $(n-1)$-dimensional combinatorial sphere. By $\Xi(L)$ we denote the space of all embeddings $\iota:|\cone L|\hookrightarrow\R^n$ such that $\iota$ takes the cone vertex to the origin and the restriction of $\iota$ to every simplex of~$\cone L$ is linear. Suppose that the space $\Xi(L)$ is non-empty. The group~$GL_n(\R)$ acts naturally on~$\Xi(L)$. The orbit space $\Sigma(L)=\Xi(L)/GL_n(\R)$ is called the \textit{configuration space} of the combinatorial sphere~$L$.

It is easy to prove that the space~$\Sigma(L)$ is contractible if $\dim L=1$. It is also known that the space~$\Sigma(L)$ is arcwise connected~\cite{Cai44} and simply connected~\cite{Ho73} if $\dim L=2$. Almost nothing is known about the spaces~$\Sigma(L)$ for $\dim L\ge 3$. In particular, it is not known, whether the space $\Sigma(L)$ is always connected if~$\dim
L=3$.

The space $\Sigma(L)$ has natural stratification by degeneracies of the configuration
$$
\Sigma_0(L)\subset\Sigma_1(L)\subset\ldots\subset\Sigma(L).
$$
The space $\Sigma_0(L)$ consists of all orbits of embeddings~$\iota\in\Xi(L)$ such that the vectors $\iota(v)$, $v\in
V(L)$, are in general position.The space $\Sigma_1(L)$ consists of all orbits of embeddings~$\iota\in\Xi(L)$ such that the set of vectors $\iota(v)$, $v\in V(L)$, contains not more than one subset of $n$ linearly dependent vectors.

Suppose that $\iota\in \Xi(L)$ and $\rho$ is a simplex of~$L$ such that
$\codim\rho=k$. We denote by $W_{\rho}\subset\R^n$ the
$(n-k)$-dimensional vector subspace spanned by vectors
$\iota(v)$, where $v$ runs over all vertices of the simplex~$\rho$. We choose an isomorphism $\alpha:\R^n/W_{\rho}\to\R^k$ and consider the composite mapping
$$
\iota_{\rho}:|\cone\Lk\rho|\subset |\cone
L|\stackrel{\iota}{\hookrightarrow} \R^n\to
\R^n/W_{\rho}\stackrel{\alpha}{\to}\R^k.
$$
Obviously, $\iota_{\rho}\in\Xi(\Lk\rho)$. Besides, the
$GL_k(\R)$-orbit of the embedding~$\iota_{\rho}$ is independent on the choice of~$\alpha$. Hence we obtain a well-defined mapping
$\varphi_{L,\rho}:\Sigma(L)\to\Sigma(\Lk\rho)$. We shall need the following particular case of this construction. Suppose $K$ is a Brouwer manifold and $\sigma$ and $\tau$ are simplices of~$K$ of codimensions~$4$ and~$3$ respectively such that
$\sigma\subset\tau$; then the mapping 
$$\varphi_{\sigma\tau}=\varphi_{\Lk\sigma,\tau\setminus\sigma}:
\Sigma(\Lk\sigma)\to \Sigma(\Lk\tau)$$
is well defined.

To apply the Gabrielov--Gelfand--Losik formula to a Brouwer manifold~$K$ we should endow this manifold with a certain additional combinatorial structure~$\B$ called a \textit{hypersimplicial system}. There is an easy combinatorial construction that, for a given hypersimplicial system, produces the following cohomology classes of configuration spaces. This construction produces, for every codimension~$4$ simplex $\sigma\in K$,  a zero-dimensional cohomology class $\theta_{\sigma}\in
H^0(\Sigma_0(\Lk\sigma);\Q)$ and, for every codimensional~$3$ simplex  
$\tau\in K$, a one-dimensional cohomology class $\theta_{\tau}\in
H^1(\Sigma_1(\Lk\tau),\Sigma_0(\Lk\tau);\Q)$. We omit the definition of a hypersimplicial system and a construction of the cohomology classes~$\theta_{\sigma}$ and $\theta_{\tau}$.

The cycle~$\Gamma$ representing the Poincar\'e dual of the first rational Pontryagin class of an oriented Brouwer manifold~$K$ can now be computed in the following way.

1. Endow the manifold~$K$ with a hypersimplicial system. From this hypersimplicial system, compute the cohomology classes~$\theta_{\sigma}$ and~$\theta_{\tau}$.

2. For every codimension~$4$ simplex~$\sigma\in K$, choose a configuration $y_{\sigma}\in\Sigma_0(\Lk\sigma)$; for any codimension~$3$ simplex~$\tau\in K$, choose a configuration $y_{\tau}\in\Sigma_0(\Lk\tau)$; for every simplices  $\sigma\subset\tau$ of codimensions~$4$ and~$3$ respectively choose a curve
$y_{\sigma\tau}:[0,1]\to\Sigma_1(\Lk\tau)$ such that
$y_{\sigma\tau}(0)=y_{\tau}$ and
$y_{\sigma\tau}(1)=\varphi_{\sigma\tau}(y_{\sigma})$.

3. Now the required cycle~$\Gamma$ is given by
$$
\Gamma=\sum_{\sigma\in K,\,\codim\sigma=4}\left(
\left\langle\theta_{\sigma},y_{\sigma}\right\rangle+
\sum_{\tau\supset\sigma,\,\codim\tau=3}
\left\langle\theta_{\tau},y_{\sigma\tau}\right\rangle\right).
$$

This formula is not local. The coefficient of a simplex in  the cycle obtained depends not only on the combinatorial structure of the link of this simplex but also on the choice of several additional structures, namely, the hypersimplicial system, the points~$y_{\sigma}$ and~$y_{\tau}$, and the curves~$y_{\sigma\tau}$. In~\cite{GGL76}, a local formula is obtained by a certain averaging procedure over different choices of these additional structures. However, this procedure makes the formula much more complicated.  

\section{Algebra $\T_*$ and local formulae}
\label{section_locform}

For each $n\ge 1$, denote by $\T_n$ the Abelian group presented by generators~$\langle L\rangle$, where $L$ runs over all oriented $(n-1)$-dimensional combinatorial spheres, and relations $\langle L_1\rangle=\langle L_2\rangle$ whenever
$L_1\cong L_2$ and $\langle -L\rangle=-\langle L\rangle$. Obviously, the group~$\T_1$ is a cyclic group of order~$2$, the group~$\T_2$ is a direct sum of countably many cyclic subgroups of order~$2$, and the group~$\T_n$ is a direct sum of countably many cyclic subgroups of order~$2$ and countably many infinite cyclic subgroups for $n\ge 3$. The summands~$\Z_2$ correspond to isomorphism classes of combinatorial spheres that possess orientation reversing automorphisms; the summands~$\Z$ correspond to isomorphism classes of combinatorial spheres that do not possess orientation reversing automorphisms. (More precisely, the summands~$\Z$ correspond to pairs of such isomorphism classes differing by the orientation.) We put,~$\T_0=\Z$.

Define a grading decreasing differential $\partial:\T_n\to
\T_{n-1}$ by
$$
\partial\langle L\rangle=\sum_{v\in V(L)}\langle\Lk v\rangle,
$$
where the links of vertices are endowed with the induced orientations. The differential $\partial:\T_1\to\T_0$ is defined to be the zero homomorphism. It is easy to check that $\partial^2=0$.

The direct sum
$$
\T_*=\bigoplus_{n=0}^{\infty}\T_n
$$
is a supercommutative associative differential graded algebra (with the differential decreasing the grading) with respect to the multiplication given by
$$
\langle L_1\rangle\langle L_2\rangle=\langle L_1*L_2\rangle.
$$
It can be checked immediately that the Leibniz formula 
$$
\partial(\lambda\mu)=(\partial\lambda)\mu+(-1)^l\lambda\partial\mu
$$
holds, where $\lambda\in\T_l$ and $\mu\in\T_m$.

Let $\Lambda$ be a commutative ring with unit. We put
$$
\T^n(\Lambda)=\Hom(\T_n,\Lambda).
$$
Then
$$
\T^*(\Lambda)=\bigoplus_{n=0}^{\infty}\T^n(\Lambda)
$$
is a supercocommutative coassociative differential graded coalgebra with the differential increasing the grading. The differential is given by
$$
(\delta f)\bigl(\langle L\rangle\bigr)=(-1)^n\sum_{v\in
V(L)}f\bigl(\langle\Lk v\rangle\bigr),
$$
where $f\in\T^n(\Lambda)$.

Suppose that $f\in \T^n(\Lambda)$ and $K$ is an oriented $m$-dimensional combinatorial manifold. We define a simplicial chain 
$f_{\sharp}(K)\in C_{m-n}(K;\Lambda)$ by
\begin{equation}\label{eq_uniform}
f_{\sharp}(K)=\sum_{\sigma\in
K,\,\dim\sigma=m-n}f\bigl(\langle\Lk\sigma\rangle\bigr)\sigma.
\end{equation}
Here every simplex $\sigma$ is endowed with an arbitrary orientation; then the orientation of~$K$ induces the orientation of the link of~$\sigma$. The sign of the summand $f\bigl(\langle\Lk\sigma\rangle\bigr)\sigma$ is independent of the chosen orientation of~$\sigma$.

The first result on the existence of local formulae for the polynomials in rational Pontryagin classes was obtained in~1975 by N.~Levitt and C.~Rourke. In our notation this result can be formulated in the following way.

\begin{theorem}
\label{theorem_LR} Suppose $F\in \Q[p_1,p_2,\ldots]$ is a homogeneous polinomial of degree $n=4k$, where $\deg p_i=4i$. Then, for any  $m\ge
n$, there exists a function $f_m\in \T^n(\Q)$ such that, for every $m$-dimensional oriented combinatorial manifold~$K$, the chain 
$f_{m\sharp}(K)$ is a cycle representing the Poincar\'e dual of the cohomology class $F\bigl(p_1(K),p_2(K),\ldots\bigr)$,
where $p_i(K)$ are the rational Pontryagin classes of~$K$.
\end{theorem}

Notice that, in the proof due to Levitt and Rourke, the functions $f_m$ for the same polynomial~$F$ but for different numbers~$m$ are not related to each other. In~\cite{Gai04} the author has proved that indeed all functions~$f_m$ can be chosen to be equal to each other. Moreover, if some function $f_m\in\T^n(\Q)$ satisfies the requirements of Theorem~\ref{theorem_LR} for some~$m\ge n$, then the same function satisfies the requirements of Theorem~\ref{theorem_LR} for all~$m\ge
n$. Thus we obtain the following theorem.

\begin{theorem}
\label{theorem_exist} Suppose $F\in \Q[p_1,p_2,\ldots]$ is a homogeneous polynomial of degree~$n=4k$, where $\deg p_i=4i$. Then there exists a function~$f\in \T^n(\Q)$ such that, for every oriented combinatorial manifold~$K$, $\dim K\ge n$, the chain $f_{\sharp}(K)$ is a cycle representing the Poincar\'e dual of the cohomology class
$F\bigl(p_1(K),p_2(K),\ldots\bigr)$.
\end{theorem}

This improvement of the Levitt--Rourke theorem is important for us because it allows to give the following definition.

\begin{defin}
A function $f\in\T^n(\Q)$ is called a (\textit{universal})
\textit{local formula} for a homogeneous polynomial  $F\in
\Q[p_1,p_2,\ldots]$ of degree $n$ if, for every oriented combinatorial manifold~$K$ such that $\dim K\ge n$, the chain 
$f_{\sharp}(K)$ is a cycle representing the Poincar\'e dual of the cohomology class $F\bigl(p_1(K),p_2(K),\ldots\bigr)$.
\end{defin}

Theorem~\ref{theorem_exist} provides the existence of local formulae for all homogeneous polynomials in rational Pontryagin classes. Now let us consider in some sense an inverse question. Let $f\in\T^n(\Q)$ be a function such that the chain
$f_{\sharp}(K)$ is a cycle for every oriented combinatorial manifold~$K$,  $\dim K\ge n$. What can we say about the homology classes represented by the cycles
$f_{\sharp}(K)$? The answer to this question has been obtained by the author in~\cite{Gai04}.

\begin{theorem}\label{theorem_locform}
For a function $f\in \T^n(\Q)$ the following three condidtions are equivalent
\begin{enumerate}
\item the chain $f_{\sharp}(K)$ is a cycle for every oriented combinatorial manifold~$K$ such that $\dim
K\ge n$;

\item $f$ is a local formula for some homogeneous polynomial in rational Pontryagin classes;

\item $f$ is a cocycle of the complex~$\T^*(\Q)$, that is, $\delta f=0$.
\end{enumerate}
Besides, the cocylces of~$\T^*(\Q)$ representing the same cohomology class are local formulae for the same polynomial in rational Pontryagin classes.
\end{theorem}
\begin{cor} The mapping taking each cocycle~$f$ of~$\T^*(\Q)$ to the polynomial~$F$ such that $f$ is a local formula for~$F$ induces an additive homomorphism
$$
\varphi:H^*\bigl(\T^*(\Q)\bigr)\to \Q[p_1,p_2,\ldots].
$$
\end{cor}

Let us now describe the homology of the differential graded algebra
$\T_*\otimes\Q$ and the cohomology of the coalgebra~$\T^*(\Q)$. Denote by~$\Omega^{SO}_*$ the oriented cobordism ring. It is well known that the ring $\Omega^{SO}_*\otimes\Q$ is a polynomial ring with one generator in every dimension divisible by~$4$.

We construct a homomorphism $\alpha:\Omega^{SO}_*\to H_*(\T_*)$ in the following way. Let $M^n$ be an oriented smooth manifold. Choose an arbitrary smooth triangulation~$K$ of~$M^n$ and put
$$
\alpha\bigl([M^n]\bigr)=\left[\sum_{v\in V(K)}\langle\Lk
v\rangle\right],
$$
where the square brackets in the left hand side denote the cobordism class and the square brackets in the right hand side denote the homology class. It can be immediately checked that the sum inside the square brackets in the right hand side is a cycle of the complex~$\T_*$ and the homology class of this cycle is independent of the triangulation~$K$. Besides, this homology class does not change if we replace the manifold~$M^n$ with a manifold cobordant to it (for details, see~\cite{Gai08}). Thus $\alpha$ is a well defined homomorphism. It is easy to check that the homomorphism~$\alpha\otimes\Q$ is conjugate to the homomorphism~$\varphi$ with respect to the canonical non-degenerated pairing
$$
\Q[p_1,p_2,\ldots]\otimes(\Omega^{SO}_*\otimes\Q)\to\Q,
$$
given by the Pontryagin numbers. In~\cite{Gai08}, the author proved the following theorem.

\begin{theorem}\label{theorem_isom}
The kernel and the cokernel of~$\alpha$ are torsion groups.
The homomorphisms 
\begin{gather*}
\alpha\otimes\Q:\Omega^{SO}_*\otimes\Q\to H_*(\T_*)\otimes\Q;\\
\varphi:H^*\bigl(\T^*(\Q)\bigr)\to \Q[p_1,p_2,\ldots]
\end{gather*}
are isomorphisms.
\end{theorem}

The following uniqueness theorem obtained in~\cite{Gai04} is a straightforward consequence of Theorem~\ref{theorem_isom}. 

\begin{theorem}\label{theorem_unique}
A local formula for a homogeneous polynomial  $F\in
\Q[p_1,p_2,\ldots]$ is unique up to a coboundary of the complex~$\T^*(\Q)$.
\end{theorem}

Now let us apply the obtained theorems to local formulae for the first rational Pontryagin class. Theorems~\ref{theorem_exist},~\ref{theorem_locform}, and~\ref{theorem_unique} imply that the function  $f\in \T^4(\Q)$ satisfying the equation $\delta f=0$ is unique up to adding a coboundary and multiplying by a rational constant. Any such function is a local formula for the first Pontryagin class multiplied by some rational constant. Our further intention is to use the technique of bistellar moves to describe explicitly all solutions~$f\in \T^4(\Q)$ of the equation~$\delta f=0$.

\section{Bistellar moves and graphs~$\Gamma_n$}
\label{section_bistellar} Let $K$ be an $n$-dimensional combinatorial manifold on the vertex set~$V$. Assume that the complex~$K$ contains a full subcomplex $\sigma^k*\partial\tau^{n-k}$, where $\sigma^k$ is a simplex of~$K$ and $\tau^{n-k}$ is an ``empty simplex'' of~$K$, that is, an $(n-k+1)$-element subset of~$V$ such that $\tau^{n-k}$ does not belong to~$K$ and all proper subsets of~$\tau^{n-k}$ belong to~$K$. Replace the full subcomplex~$\sigma^k*\partial\tau^{n-k}\subset K$ by the full subcomplex~$\partial\sigma^k*\tau^{n-k}$ and denote the obtained simplicial complex by~$K_1$. It can be easily checked that $K_1$ is a combinatorial manifold piecewise linearly homeomorphic to~$K$. The described operation is called a  {\it bistellar move} and is denoted by~$\beta=\beta_{K,\sigma^k}$. The obtained combinatorial manifold~$K_1$ is denoted by~$\beta(K)$. In the described construction we can suppose that $k=0$ or $k=n$ using the agreements $\partial\,\pt=\varnothing$ and
$\sigma*\varnothing=\sigma$. Thus particular cases of bistellar moves are 
stellar subdivisions of $n$-dimensional simplices and inverse operations.
The bistellar move $\beta_{K_1,\tau^{n-k}}$ is said to be the bistellar move  \textit{inverse} to~$\beta$ and is denoted by~$\beta^{-1}$. All types of bistellar moves for manifolds of dimensions~$2$ and~$3$ are shown in Fig.~\ref{fig_bm2} and~\ref{fig_bm3} respectively. The move shown in Fig.~\ref{fig_bm3} on the right transforms two tetrahedra with common two-dimensional face to three tetrahedra with common edge.

\begin{figure}
{\unitlength=0.3mm
\begin{picture}(370,60)

\thicklines

\put(0,5){\circle*{1}}

\put(100,5){\circle*{1}}

\put(130,25){\circle*{1}}

\put(30,55){\circle*{1}}

\put(60,5){\circle*{1}}

\put(130,55){\circle*{1}}

\put(160,5){\circle*{1}}

\put(210,30){\circle*{1}}

\put(240,0){\circle*{1}}

\put(240,60){\circle*{1}}

\put(270,30){\circle*{1}}

\put(310,30){\circle*{1}}

\put(340,0){\circle*{1}}

\put(340,60){\circle*{1}}

\put(370,30){\circle*{1}}

\put(0,5){\line(3,5){30}}

\put(30,55){\line(3,-5){30}}

\put(0,5){\line(1,0){60}}

\put(100,5){\line(3,5){30}}

\put(130,55){\line(3,-5){30}}

\put(100,5){\line(1,0){60}}

\put(100,5){\line(3,2){30}}

\put(130,25){\line(0,1){30}}

\put(130,25){\line(3,-2){30}}

\put(210,30){\line(1,1){30}}

\put(240,60){\line(1,-1){30}}

\put(210,30){\line(1,-1){30}}

\put(240,0){\line(1,1){30}}

\put(240,0){\line(0,1){60}}

\put(310,30){\line(1,1){30}}

\put(340,60){\line(1,-1){30}}

\put(310,30){\line(1,-1){30}}

\put(340,0){\line(1,1){30}}

\put(310,30){\line(1,0){60}}

\put(276,30){\vector(1,0){28}}

\put(304,30){\vector(-1,0){28}}

\put(66,30){\vector(1,0){28}}

\put(94,30){\vector(-1,0){28}}

\end{picture}
}

\caption{Bistellar moves for manifolds of dimension~$2$}\label{fig_bm2}
\bigskip
\bigskip
{\unitlength=0.3mm
\begin{picture}(370,100)

\thicklines


\put(0,40){\circle*{1}}

\put(20,20){\circle*{1}}

\put(30,80){\circle*{1}}

\put(60,30){\circle*{1}}

\put(100,40){\circle*{1}}

\put(120,20){\circle*{1}}

\put(130,80){\circle*{1}}

\put(160,30){\circle*{1}}

\put(130,40){\circle*{1}}

\put(0,40){\line(3,4){30}}

\thinlines

\put(0,40){\line(6,-1){60}}

\thicklines

\put(0,40){\line(1,-1){20}}

\put(20,20){\line(4,1){40}}

\put(20,20){\line(1,6){10}}

\put(30,80){\line(3,-5){30}}

\put(100,40){\line(3,4){30}}

\put(100,40){\line(1,-1){20}}

\put(120,20){\line(4,1){40}}

\put(120,20){\line(1,6){10}}

\put(130,80){\line(3,-5){30}}

\thinlines

\put(100,40){\line(6,-1){60}}

\put(100,40){\line(1,0){30}}

\put(130,40){\line(0,1){40}}

\put(130,40){\line(3,-1){30}}

\put(130,40){\line(-1,-2){10}}

\thicklines

\put(66,50){\vector(1,0){28}}

\put(94,50){\vector(-1,0){28}}


\put(210,60){\circle*{1}}

\put(230,40){\circle*{1}}

\put(240,100){\circle*{1}}

\put(270,50){\circle*{1}}

\put(240,0){\circle*{1}}

\put(310,60){\circle*{1}}

\put(330,40){\circle*{1}}

\put(340,100){\circle*{1}}

\put(370,50){\circle*{1}}

\put(340,0){\circle*{1}}

\put(210,60){\line(3,4){30}}

\thinlines

\put(210,60){\line(6,-1){60}}

\thicklines

\put(210,60){\line(1,-1){20}}

\put(230,40){\line(4,1){40}}

\put(230,40){\line(1,6){10}}

\put(240,100){\line(3,-5){30}}

\put(210,60){\line(1,-2){30}}

\put(230,40){\line(1,-4){10}}

\put(240,0){\line(3,5){30}}

\put(310,60){\line(3,4){30}}

\put(310,60){\line(1,-1){20}}

\put(330,40){\line(4,1){40}}

\put(330,40){\line(1,6){10}}

\put(340,100){\line(3,-5){30}}

\put(310,60){\line(1,-2){30}}

\put(330,40){\line(1,-4){10}}

\put(340,0){\line(3,5){30}}

\thinlines

\put(310,60){\line(6,-1){60}}

\put(340,0){\line(0,1){100}}

\thicklines

\put(276,50){\vector(1,0){28}}

\put(304,50){\vector(-1,0){28}}

\end{picture}
}

\caption{Bistellar moves for manifolds of dimension~$3$}\label{fig_bm3}
\end{figure}

Suppose that $K_1$ and $K_2$ are oriented combinatorial manifolds of the same dimension and $\sigma_1\in K_1$ and $\sigma_2\in K_2$ are simplices such that there exist bistellar moves
$\beta_{K_1,\sigma_1}$ and $\beta_{K_2,\sigma_2}$. The bistellar moves~$\beta_{K_1,\sigma_1}$ and $\beta_{K_2,\sigma_2}$ are said to be \textit{equivalent} if there exists an isomorphism
$f:K_1\to K_2$ such that $f(\sigma_1)=\sigma_2$. A bistellar move $\beta$ is said to be \textit{inessential} if
$\beta$ is equivalent to~$\beta^{-1}$. All other bistellar moves are said to be \textit{essential}.

If $k\ne 0,n$, then the sets of vertices of combinatorial manifolds~$K$ and~$K_1$ coincide; if either $k=0$ or $k=n$, then either the set~$V(K)$ or the set~$V(K_1)$ contains one superfluous vertex. In the first case we put $V(\beta)=V(K)=V(K_1)$; in the second case we denote by $V(\beta)$ the larger of the sets~$V(K)$ and~$V(K_1)$. Besides, we put
$U(\beta)=V\bigl(\partial\sigma^k*\partial\tau^{n-k}\bigr)\subset
V(\beta)$. The set $U(\beta)$ consists of those vertices that neither appear nor disappear under the bistellar move~$\beta$ and whose links change under~$\beta$. For any vertex~$v\in U(\beta)$ the bistellar move~$\beta$ induces the bistellar move
$\beta_v=\beta_{\Lk_Kv,\sigma^k\setminus\{v\}}$ transforming the combinatorial sphere~$\Lk_Kv$ to the combinatorial sphere~$\Lk_{K_1}v$.

In 1987, Pachner~\cite{Pac87} (see also~\cite{Pac91})
proved the following theorem.

\begin{theorem}
Suppose $K_1$ and $K_2$ are piecewise linearly homeomorphic combinatorial manifolds. Then $K_1$ can be transformed to~$K_2$ by a finite sequence of bistellar moves and isomorphisms.
\end{theorem}

In particular, any two combinatorial spheres of the same dimension can be transformed to each other by a sequence of bistellar moves and isomorphisms. Therefore to decribe functions
$f\in\T^4(\Q)$ such that $\delta f=0$ we suffice to describe how the value~$f\bigl(\langle L\rangle\bigr)$ changes under bistellar moves of a three-dimensional combinatorial sphere~$L$. We shall need the following additional constructions.

For any $n$, we construct an infinite graph~$\Gamma_n$ in the following way. Vertices of~$\Gamma_n$ correspond to isomorphism classes of oriented $n$-dimensional combinatorial spheres. Both the isomorphism class of a combinatorial sphere~$L$ and the corresponding vertex of the graph~$\Gamma_n$ are denoted by~$\{L\}$. We stress that  $\{L\}$ and
$\{-L\}$ are distinct vertices unless $L$ possesses an orientation reversing automorphism. Edges of~$\Gamma_n$ correspond to equivalence classes of essential bistellar moves of $n$-dimensional combinatorial spheres. The edges corresponding to the equivalence classes of bistellar moves $\beta$ and $\beta^{-1}$ coincide but have opposite orientations.
(The graph~$\Gamma_n$ can contain multiple edges and loops.) The oriented edge corresponding to the equivalence class of a bistellar move~$\beta$ will be denoted by~$\{\beta\}$. Thus, $\{\beta^{-1}\}=-\{\beta\}$.

The group~$\Z_2$ acts on the graph~$\Gamma_n$ by reversing the orientations of combinatorial spheres. By $\CQ$ we denote the group $\Q$ regarded as a  $\Z_2$-module such that the generator of
$\Z_2$ acts by multiplication by~$-1$. By $C_{\Z_2}^i(\Gamma_n;\CQ)$ and $H_{\Z_2}^i(\Gamma_n;\CQ)$, $i=0,1$, we denote the $i$-dimensional equivariant cellular cochain groups and the $i$-dimensional equivariant cohomology groups of~$\Gamma_n$ with respect to the described actions of the group~$\Z_2$. The differential of the cochain complex $C_{\Z_2}^*(\Gamma_n;\CQ)$ will be denoted by~$d$. Pachner's theorem implies that the graph $\Gamma_n$ is connected.
Hence, $H_{\Z_2}^0(\Gamma_n;\CQ)=0$. Obviously, the group
$C_{\Z_2}^0(\Gamma_n;\CQ)$ is canonically isomorphic to the group~$\T^{n+1}(\Q)$. Therefore we have the differential
$$
\delta:C_{\Z_2}^0(\Gamma_n;\CQ)\to C_{\Z_2}^0(\Gamma_{n+1};\CQ).
$$
Define the differential
$$
\delta:C_{\Z_2}^1(\Gamma_n;\CQ)\to C_{\Z_2}^1(\Gamma_{n+1};\CQ)
$$
by
$$
(\delta h)\bigl(\{\beta\}\bigr)=(-1)^n\sum_{v\in
U(\beta)}h\bigl(\{\beta_v\}\bigr).
$$

It is easy to check that $d\delta=\delta d$. Hence the monomorphism~$d$ is a chain mapping of the cochain complex~$\bigl(C_{\Z_2}^0(\Gamma_*;\CQ),\delta\bigr)$ to the cochain complex~$\bigl(C_{\Z_2}^1(\Gamma_*;\CQ),\delta\bigr)$. This chain mapping is null-homotopic. The chain homotopy can be constructed in the following way. Let $\beta$ be a bistellar move transforming an oriented $(n-1)$-dimensional sphere~$L_1$ to a combinatorial sphere~$L_2$ and replacing a full subcomplex
$\sigma*\partial\tau\subset L_1$ with a full subcomplex
$\partial\sigma*\tau\subset L_2$. Consider the abstract simplicial complex~$L_{\beta}$ on the vertex set 
$V(\beta)\sqcup\{u_1,u_2\}$ consisting of all simplices $\rho\in
L_1\cup L_2$, all simplices $\rho\cup\{u_1\}$ such that
$\rho\in L_1$, all simplices $\rho\cup\{u_2\}$ such that $\rho\in L_2$, and the simplex~$\sigma\cup\tau$. It can be immediately checked that the complex~$L_{\beta}$ is an $n$-dimensional combinatorial sphere. Orient the sphere~$L_{\beta}$ so that the natural isomorphism between the link of the vertex $u_2$ and the combinatorial sphere~$L_2$ preserves the orientation. Then the natural isomorphism between the link of the vertex~$u_1$ and~$L_1$ reverses the orientation.
Define the homomorphism
$$s:C_{\Z_2}^0(\Gamma_n;\CQ)\to
C_{\Z_2}^1(\Gamma_{n-1};\CQ)$$ by
$$
s(f)\bigl(\{\beta\}\bigr)=(-1)^{n-1}f\bigl(\{ L_{\beta}\}\bigr).
$$
It can be easily checked that $d=\delta s-s\delta$.

Let us now consider in more details the cases~$n=1$ and~$n=2$. The graph~$\Gamma_1$ is isomorphic to the graph with vertices indexed by natural numbers not less than~$3$ and a unique edge connecting the vertices~$k$ and~$k+1$ for every~$k\ge 3$. The group
$\Z_2$ acts on $\Gamma_1$ trivially. Therefore,
$C_{\Z_2}^0(\Gamma_1;\CQ)=C_{\Z_2}^1(\Gamma_1;\CQ)=0$.

In the graph~$\Gamma_2$, we consider two types of cycles, which we shall call \textit{elementary cycles}. The first type of elementary cycles corresponds to the ``commutation'' of two bistellar moves. Suppose $L$ is an oriented two-dimensional combinatorial sphere and $\sigma_1$ and $\sigma_2$ are simplices of~$L$ such that

1) no simplex of~$L$ contains both~$\sigma_1$ and~$\sigma_2$,

2) $\Lk\sigma_i=\partial\tau_i$, for some ``empty simplices''
$\tau_i$, $i=1,2$, and

3) there exist bistellar moves $\beta_1=\beta_{L,\sigma_1}$,
$\beta_2=\beta_{L_2,\sigma_2}$, $\beta_3=\beta_{L_3,\tau_1}$, and
$\beta_4=\beta_{L_4,\tau_2}$, where $L_2=\beta_1(L)$,
$L_3=\beta_2(L_2)$ и $L_4=\beta_3(L_3)$.

Then we have
$\beta_4(L_4)\cong L$. Therefore we obtain a cycle consisting of edges  $\{\beta_i\}$,
$i=1,2,3,4$, in the graph~$\Gamma_2$. We denote this cycle by
$\gamma(L,\sigma_1,\sigma_2)$. Different elementary cycles of the first type are shown in Fig.~\ref{fig_comm},~\textit{a--i}. Elementary cycles of the second type are the cycles in the graph~$\Gamma_2$ shown in Fig.~\ref{fig_dop},~\textit{a--c}. For elementary cycles of both the first and the second types one should omit all inessential bistellar moves.

In the sequel, under a cycle of a graph we always mean a cellular $1$-chain with zero boundary rather than a closed sequence of edges.

\begin{figure}[t]

{\unitlength=0.5mm

\begin{picture}(240,400)

\thicklines

\put(2,332){\begin{picture}(102,56)


\put(0,39){
     \begin{picture}(47,16)
     \put(1,1){\circle*{1}}
     \put(11,16){\circle*{1}}
     \put(21,1){\circle*{1}}
     \put(26,1){\circle*{1}}
     \put(36,16){\circle*{1}}
     \put(46,1){\circle*{1}}
     \put(1,1){\line(1,0){20}}
     \put(1,1){\line(2,3){10}}
     \put(11,16){\line(2,-3){10}}
     \put(26,1){\line(1,0){20}}
     \put(26,1){\line(2,3){10}}
     \put(36,16){\line(2,-3){10}}
     \end{picture}
    }
\put(55,39){
     \begin{picture}(47,16)
     \put(1,1){\circle*{1}}
     \put(11,16){\circle*{1}}
     \put(21,1){\circle*{1}}
     \put(26,1){\circle*{1}}
     \put(36,16){\circle*{1}}
     \put(46,1){\circle*{1}}
     \put(1,1){\line(1,0){20}}
     \put(1,1){\line(2,3){10}}
     \put(11,16){\line(2,-3){10}}
     \put(26,1){\line(1,0){20}}
     \put(26,1){\line(2,3){10}}
     \put(36,16){\line(2,-3){10}}
     \end{picture}
    }
\put(55,9){
     \begin{picture}(47,16)
     \put(1,1){\circle*{1}}
     \put(11,16){\circle*{1}}
     \put(21,1){\circle*{1}}
     \put(26,1){\circle*{1}}
     \put(36,16){\circle*{1}}
     \put(46,1){\circle*{1}}
     \put(1,1){\line(1,0){20}}
     \put(1,1){\line(2,3){10}}
     \put(11,16){\line(2,-3){10}}
     \put(26,1){\line(1,0){20}}
     \put(26,1){\line(2,3){10}}
     \put(36,16){\line(2,-3){10}}
     \end{picture}
    }
\put(0,9){
     \begin{picture}(47,16)
     \put(1,1){\circle*{1}}
     \put(11,16){\circle*{1}}
     \put(21,1){\circle*{1}}
     \put(26,1){\circle*{1}}
     \put(36,16){\circle*{1}}
     \put(46,1){\circle*{1}}
     \put(1,1){\line(1,0){20}}
     \put(1,1){\line(2,3){10}}
     \put(11,16){\line(2,-3){10}}
     \put(26,1){\line(1,0){20}}
     \put(26,1){\line(2,3){10}}
     \put(36,16){\line(2,-3){10}}
     \end{picture}
    }
\put(47.5,47){\vector(1,0){10}} \put(81,35){\vector(0,-1){10}}
\put(57.5,17){\vector(-1,0){10}} \put(26,25){\vector(0,1){10}}
\put(56,40){
    \begin{picture}(20,15)
    \put(10,5){\circle*{1}}
    \put(0,0){\line(2,1){10}}
    \put(10,5){\line(0,1){10}}
    \put(10,5){\line(2,-1){10}}
    \end{picture}
    }
\put(56,10){
    \begin{picture}(20,15)
    \put(10,5){\circle*{1}}
    \put(0,0){\line(2,1){10}}
    \put(10,5){\line(0,1){10}}
    \put(10,5){\line(2,-1){10}}
    \end{picture}
    }
\put(81,10){
    \begin{picture}(20,15)
    \put(10,5){\circle*{1}}
    \put(0,0){\line(2,1){10}}
    \put(10,5){\line(0,1){10}}
    \put(10,5){\line(2,-1){10}}
    \end{picture}
    }
\put(26,10){
    \begin{picture}(20,15)
    \put(10,5){\circle*{1}}
    \put(0,0){\line(2,1){10}}
    \put(10,5){\line(0,1){10}}
    \put(10,5){\line(2,-1){10}}
    \end{picture}
    }
\end{picture}
}


\put(132,332){\begin{picture}(102,56) \put(0,39){
     \begin{picture}(47,16)
     \put(1,1){\circle*{1}}
     \put(1,16){\circle*{1}}
     \put(19.75,8.5){\circle*{1}}
     \put(38.5,16){\circle*{1}}
     \put(38.5,1){\circle*{1}}
     \put(1,1){\line(0,1){15}}
     \put(1,1){\line(5,2){18.75}}
     \put(1,16){\line(5,-2){18.75}}
     \put(38.5,1){\line(0,1){15}}
     \put(38.5,1){\line(-5,2){18.75}}
     \put(38.5,16){\line(-5,-2){18.75}}
     \qbezier(22.5,9.9)(19.75,12)(16.25,9.9)
     \put(18,14.5){$p$}
     \qbezier(22.5,7.1)(19.75,5)(16.25,7.1)
     \put(18,1){$q$}
     \end{picture}
    }
\put(62.5,39){
     \begin{picture}(47,16)
     \put(1,1){\circle*{1}}
     \put(1,16){\circle*{1}}
     \put(19.75,8.5){\circle*{1}}
     \put(38.5,16){\circle*{1}}
     \put(38.5,1){\circle*{1}}
     \put(4.75,8.5){\circle*{1}}
     \put(1,1){\line(0,1){15}}
     \put(1,1){\line(5,2){18.75}}
     \put(1,16){\line(5,-2){18.75}}
     \put(38.5,1){\line(0,1){15}}
     \put(38.5,1){\line(-5,2){18.75}}
     \put(38.5,16){\line(-5,-2){18.75}}
     \put(1,1){\line(1,2){3.75}}
     \put(1,16){\line(1,-2){3.75}}
     \put(4.75,8.5){\line(1,0){15}}
     \end{picture}
    }
\put(62.5,9){
     \begin{picture}(47,16)
     \put(1,1){\circle*{1}}
     \put(1,16){\circle*{1}}
     \put(19.75,8.5){\circle*{1}}
     \put(38.5,16){\circle*{1}}
     \put(38.5,1){\circle*{1}}
     \put(4.75,8.5){\circle*{1}}
     \put(34.75,8.5){\circle*{1}}
     \put(1,1){\line(0,1){15}}
     \put(1,1){\line(5,2){18.75}}
     \put(1,16){\line(5,-2){18.75}}
     \put(38.5,1){\line(0,1){15}}
     \put(38.5,1){\line(-5,2){18.75}}
     \put(38.5,16){\line(-5,-2){18.75}}
     \put(1,1){\line(1,2){3.75}}
     \put(1,16){\line(1,-2){3.75}}
     \put(4.75,8.5){\line(1,0){15}}
     \put(38.5,1){\line(-1,2){3.75}}
     \put(38.5,16){\line(-1,-2){3.75}}
     \put(19.75,8.5){\line(1,0){15}}
     \end{picture}
    }
\put(0,9){
     \begin{picture}(47,16)
     \put(1,1){\circle*{1}}
     \put(1,16){\circle*{1}}
     \put(19.75,8.5){\circle*{1}}
     \put(38.5,16){\circle*{1}}
     \put(38.5,1){\circle*{1}}
     \put(34.75,8.5){\circle*{1}}
     \put(1,1){\line(0,1){15}}
     \put(1,1){\line(5,2){18.75}}
     \put(1,16){\line(5,-2){18.75}}
     \put(38.5,1){\line(0,1){15}}
     \put(38.5,1){\line(-5,2){18.75}}
     \put(38.5,16){\line(-5,-2){18.75}}
     \put(38.5,1){\line(-1,2){3.75}}
     \put(38.5,16){\line(-1,-2){3.75}}
     \put(19.75,8.5){\line(1,0){15}}
     \end{picture}
    }
\put(47.5,47){\vector(1,0){10}} \put(84.5,35){\vector(0,-1){10}}
\put(57.5,17){\vector(-1,0){10}} \put(22,25){\vector(0,1){10}}


\end{picture}
}


\put(4,249){\begin{picture}(102,56) \put(0,39){
     \begin{picture}(47,16)
     \put(1,8.5){\circle*{1}}
     \put(19.75,16){\circle*{1}}
     \put(19.75,1){\circle*{1}}
     \put(38.5,8.5){\circle*{1}}
     \put(19.75,1){\line(0,1){15}}
     \put(1,8.5){\line(5,2){18.75}}
     \put(1,8.5){\line(5,-2){18.75}}
     \put(38.5,8.5){\line(-5,2){18.75}}
     \put(38.5,8.5){\line(-5,-2){18.75}}
     \qbezier(17.75,15.3)(16.75,19)(19.75,19)
     \qbezier(19.75,19)(22.75,19)(21.75,15.3)
     \put(18,22){$p$}
     \qbezier(17.75,1.7)(16.75,-2)(19.75,-2)
     \qbezier(19.75,-2)(22.75,-2)(21.75,1.7)
     \put(20.5,-5.5){$q$}
     \end{picture}
    }
\put(62.5,39){
     \begin{picture}(47,16)
     \put(1,8.5){\circle*{1}}
     \put(19.75,16){\circle*{1}}
     \put(19.75,1){\circle*{1}}
     \put(38.5,8.5){\circle*{1}}
     \put(19.75,1){\line(0,1){15}}
     \put(1,8.5){\line(5,2){18.75}}
     \put(1,8.5){\line(5,-2){18.75}}
     \put(38.5,8.5){\line(-5,2){18.75}}
     \put(38.5,8.5){\line(-5,-2){18.75}}
     \put(16,8.5){\circle*{1}}
     \put(16,8.5){\line(1,2){3.75}}
     \put(16,8.5){\line(1,-2){3.75}}
     \put(1,8.5){\line(1,0){15}}
     \end{picture}
    }
\put(62.5,9){
     \begin{picture}(47,16)
     \put(1,8.5){\circle*{1}}
     \put(19.75,16){\circle*{1}}
     \put(19.75,1){\circle*{1}}
     \put(38.5,8.5){\circle*{1}}
     \put(19.75,1){\line(0,1){15}}
     \put(1,8.5){\line(5,2){18.75}}
     \put(1,8.5){\line(5,-2){18.75}}
     \put(38.5,8.5){\line(-5,2){18.75}}
     \put(38.5,8.5){\line(-5,-2){18.75}}
     \put(16,8.5){\circle*{1}}
     \put(16,8.5){\line(1,2){3.75}}
     \put(16,8.5){\line(1,-2){3.75}}
     \put(1,8.5){\line(1,0){15}}
     \put(23.5,8.5){\circle*{1}}
     \put(23.5,8.5){\line(-1,2){3.75}}
     \put(23.5,8.5){\line(-1,-2){3.75}}
     \put(23.5,8.5){\line(1,0){15}}
\end{picture}
    }
\put(0,9){
     \begin{picture}(47,16)
     \put(1,8.5){\circle*{1}}
     \put(19.75,16){\circle*{1}}
     \put(19.75,1){\circle*{1}}
     \put(38.5,8.5){\circle*{1}}
     \put(19.75,1){\line(0,1){15}}
     \put(1,8.5){\line(5,2){18.75}}
     \put(1,8.5){\line(5,-2){18.75}}
     \put(38.5,8.5){\line(-5,2){18.75}}
     \put(38.5,8.5){\line(-5,-2){18.75}}
     \put(23.5,8.5){\circle*{1}}
     \put(23.5,8.5){\line(-1,2){3.75}}
     \put(23.5,8.5){\line(-1,-2){3.75}}
     \put(23.5,8.5){\line(1,0){15}}
     \end{picture}
    }
\put(49.5,47.5){\vector(1,0){10}} \put(85,37.5){\vector(0,-1){10}}
\put(59.5,17.5){\vector(-1,0){10}} \put(22.5,27){\vector(0,1){8}}


\end{picture}
}


\put(133,249){\begin{picture}(102,56)

\put(0,39){
     \begin{picture}(47,16)
     \put(1,1){\circle*{1}}
     \put(11,16){\circle*{1}}
     \put(21,1){\circle*{1}}
     \put(26,1){\circle*{1}}
     \put(26,16){\circle*{1}}
     \put(41,1){\circle*{1}}
     \put(41,16){\circle*{1}}
     \put(1,1){\line(1,0){20}}
     \put(1,1){\line(2,3){10}}
     \put(11,16){\line(2,-3){10}}
     \put(26,1){\line(1,0){15}}
     \put(26,1){\line(0,1){15}}
     \put(26,16){\line(1,0){15}}
     \put(41,1){\line(0,1){15}}
     \put(26,16){\line(1,-1){15}}
     \end{picture}
    }
\put(55,39){
     \begin{picture}(47,16)
     \put(1,1){\circle*{1}}
     \put(11,16){\circle*{1}}
     \put(21,1){\circle*{1}}
     \put(26,1){\circle*{1}}
     \put(26,16){\circle*{1}}
     \put(41,1){\circle*{1}}
     \put(41,16){\circle*{1}}
     \put(1,1){\line(1,0){20}}
     \put(1,1){\line(2,3){10}}
     \put(11,16){\line(2,-3){10}}
     \put(26,1){\line(1,0){15}}
     \put(26,1){\line(0,1){15}}
     \put(26,16){\line(1,0){15}}
     \put(41,1){\line(0,1){15}}
     \put(26,16){\line(1,-1){15}}
     \end{picture}
    }
\put(55,9){
     \begin{picture}(47,16)
     \put(1,1){\circle*{1}}
     \put(11,16){\circle*{1}}
     \put(21,1){\circle*{1}}
     \put(26,1){\circle*{1}}
     \put(26,16){\circle*{1}}
     \put(41,1){\circle*{1}}
     \put(41,16){\circle*{1}}
     \put(1,1){\line(1,0){20}}
     \put(1,1){\line(2,3){10}}
     \put(11,16){\line(2,-3){10}}
     \put(26,1){\line(1,0){15}}
     \put(26,1){\line(0,1){15}}
     \put(26,16){\line(1,0){15}}
     \put(41,1){\line(0,1){15}}
     \put(26,1){\line(1,1){15}}
     \end{picture}
    }
\put(0,9){
     \begin{picture}(47,16)
     \put(1,1){\circle*{1}}
     \put(11,16){\circle*{1}}
     \put(21,1){\circle*{1}}
     \put(26,1){\circle*{1}}
     \put(26,16){\circle*{1}}
     \put(41,1){\circle*{1}}
     \put(41,16){\circle*{1}}
     \put(1,1){\line(1,0){20}}
     \put(1,1){\line(2,3){10}}
     \put(11,16){\line(2,-3){10}}
     \put(26,1){\line(1,0){15}}
     \put(26,1){\line(0,1){15}}
     \put(26,16){\line(1,0){15}}
     \put(41,1){\line(0,1){15}}
     \put(26,1){\line(1,1){15}}
     \end{picture}
    }
\put(46,47){\vector(1,0){10}} \put(77,35){\vector(0,-1){10}}
\put(56,17){\vector(-1,0){10}} \put(23,25){\vector(0,1){10}}
\put(56,40){
\begin{picture}(20,15)
    \put(10,5){\circle*{1}}
    \put(0,0){\line(2,1){10}}
    \put(10,5){\line(0,1){10}}
    \put(10,5){\line(2,-1){10}}
    \end{picture}
    }
\put(56,10){
    \begin{picture}(20,15)
    \put(10,5){\circle*{1}}
    \put(0,0){\line(2,1){10}}
    \put(10,5){\line(0,1){10}}
    \put(10,5){\line(2,-1){10}}
    \end{picture}
    }

\end{picture}
}


\put(5,166){\begin{picture}(102,60)

\put(0,39){
     \begin{picture}(47,16)
     \put(1,1){\circle*{1}}
     \put(1,16){\circle*{1}}
     \put(19.75,8.5){\circle*{1}}
     \put(43.75,8.5){\circle*{1}}
     \put(31.75,20.5){\circle*{1}}
     \put(31.75,-3.5){\circle*{1}}
     \put(1,1){\line(0,1){15}}
     \put(1,1){\line(5,2){18.75}}
     \put(1,16){\line(5,-2){18.75}}
     \put(31.75,-3.5){\line(0,1){24}}

     \put(19.75,8.5){\line(1,1){12}}
     \put(19.75,8.5){\line(1,-1){12}}
     \put(31.75,-3.5){\line(1,1){12}}
     \put(31.75,20.5){\line(1,-1){12}}
     \qbezier(21,10.25)(17.5,12.2)(16.25,9.9)
     \put(17,17){$p$}
     \qbezier(21,6.75)(17.5,4.7)(16.25,7.1)
     \put(17,-1){$q$}
     \end{picture}
    }
\put(55,39){
     \begin{picture}(47,16)
     \put(1,1){\circle*{1}}
     \put(1,16){\circle*{1}}
     \put(19.75,8.5){\circle*{1}}
     \put(43.75,8.5){\circle*{1}}
     \put(31.75,20.5){\circle*{1}}
     \put(31.75,-3.5){\circle*{1}}
     \put(4.75,8.5){\circle*{1}}
     \put(1,1){\line(0,1){15}}
     \put(1,1){\line(5,2){18.75}}
     \put(1,16){\line(5,-2){18.75}}
     \put(31.75,-3.5){\line(0,1){24}}
     \put(19.75,8.5){\line(1,1){12}}
     \put(19.75,8.5){\line(1,-1){12}}
     \put(31.75,-3.5){\line(1,1){12}}
     \put(31.75,20.5){\line(1,-1){12}}
     \put(4.75,8.5){\line(1,0){15}}
     \put(1,1){\line(1,2){3.75}}
     \put(1,16){\line(1,-2){3.75}}
     \end{picture}
    }
\put(55,9){
     \begin{picture}(47,16)
     \put(1,1){\circle*{1}}
     \put(1,16){\circle*{1}}
     \put(19.75,8.5){\circle*{1}}
     \put(43.75,8.5){\circle*{1}}
     \put(31.75,20.5){\circle*{1}}
     \put(31.75,-3.5){\circle*{1}}
     \put(4.75,8.5){\circle*{1}}
     \put(1,1){\line(0,1){15}}
     \put(1,1){\line(5,2){18.75}}
     \put(1,16){\line(5,-2){18.75}}
     \put(19.75,8.5){\line(1,0){24}}
     \put(19.75,8.5){\line(1,1){12}}
     \put(19.75,8.5){\line(1,-1){12}}
     \put(31.75,-3.5){\line(1,1){12}}
     \put(31.75,20.5){\line(1,-1){12}}
     \put(4.75,8.5){\line(1,0){15}}
     \put(1,1){\line(1,2){3.75}}
     \put(1,16){\line(1,-2){3.75}}
     \end{picture}
    }
\put(0,9){
     \begin{picture}(47,16)
     \put(1,1){\circle*{1}}
     \put(1,16){\circle*{1}}
     \put(19.75,8.5){\circle*{1}}
     \put(43.75,8.5){\circle*{1}}
     \put(31.75,20.5){\circle*{1}}
     \put(31.75,-3.5){\circle*{1}}

     \put(1,1){\line(0,1){15}}
     \put(1,1){\line(5,2){18.75}}
     \put(1,16){\line(5,-2){18.75}}
     \put(19.75,8.5){\line(1,0){24}}
     \put(19.75,8.5){\line(1,1){12}}
     \put(19.75,8.5){\line(1,-1){12}}
     \put(31.75,-3.5){\line(1,1){12}}
     \put(31.75,20.5){\line(1,-1){12}}
     \end{picture}
    }
\put(47.3,47){\vector(1,0){9}} \put(77,35){\vector(0,-1){10}}
\put(56.3,17){\vector(-1,0){9}} \put(23,25){\vector(0,1){10}}

\end{picture}
}


\put(133,166){\begin{picture}(102,61)

\put(0,41){
     \begin{picture}(47,16)
     \put(1,8.5){\circle*{1}}
     \put(19.75,16){\circle*{1}}
     \put(19.75,1){\circle*{1}}
     \put(34.75,1){\circle*{1}}
     \put(34.75,16){\circle*{1}}
     \put(19.75,1){\line(0,1){15}}
     \put(1,8.5){\line(5,2){18.75}}
     \put(1,8.5){\line(5,-2){18.75}}
     \put(19.75,1){\line(1,0){15}}
     \put(19.75,16){\line(1,0){15}}
     \put(34.75,1){\line(0,1){15}}
     \put(19.75,16){\line(1,-1){15}}
     \qbezier(17.75,15.5)(17,18)(19.75,18)
     \qbezier(21.75,16)(21.75,18)(19.75,18)
     \qbezier(17.75,1.5)(17,-1)(19.75,-1)
     \qbezier(21.75,1)(21.75,-1)(19.75,-1)
     \end{picture}
    }
\put(61.25,41){
     \begin{picture}(47,16)
     \put(1,8.5){\circle*{1}}
     \put(19.75,16){\circle*{1}}
     \put(19.75,1){\circle*{1}}
     \put(34.75,1){\circle*{1}}
     \put(34.75,16){\circle*{1}}
     \put(16,8.5){\circle*{1}}
     \put(19.75,1){\line(0,1){15}}
     \put(1,8.5){\line(5,2){18.75}}
     \put(1,8.5){\line(5,-2){18.75}}
     \put(19.75,1){\line(1,0){15}}
     \put(19.75,16){\line(1,0){15}}
     \put(34.75,1){\line(0,1){15}}
     \put(19.75,16){\line(1,-1){15}}
     \put(16,8.5){\line(1,2){3.75}}
     \put(16,8.5){\line(1,-2){3.75}}
     \put(1,8.5){\line(1,0){15}}
     \end{picture}
    }
\put(61.25,9){
     \begin{picture}(47,16)
     \put(1,8.5){\circle*{1}}
     \put(19.75,16){\circle*{1}}
     \put(19.75,1){\circle*{1}}
     \put(34.75,1){\circle*{1}}
     \put(34.75,16){\circle*{1}}
     \put(16,8.5){\circle*{1}}
     \put(19.75,1){\line(0,1){15}}
     \put(1,8.5){\line(5,2){18.75}}
     \put(1,8.5){\line(5,-2){18.75}}
     \put(19.75,1){\line(1,0){15}}
     \put(19.75,16){\line(1,0){15}}
     \put(34.75,1){\line(0,1){15}}
     \put(19.75,1){\line(1,1){15}}
     \put(16,8.5){\line(1,2){3.75}}
     \put(16,8.5){\line(1,-2){3.75}}
     \put(1,8.5){\line(1,0){15}}
     \end{picture}
    }
\put(0,9){
     \begin{picture}(47,16)
     \put(1,8.5){\circle*{1}}
     \put(19.75,16){\circle*{1}}
     \put(19.75,1){\circle*{1}}
     \put(34.75,1){\circle*{1}}
     \put(34.75,16){\circle*{1}}
     \put(19.75,1){\line(0,1){15}}
     \put(1,8.5){\line(5,2){18.75}}
     \put(1,8.5){\line(5,-2){18.75}}
     \put(19.75,1){\line(1,0){15}}
     \put(19.75,16){\line(1,0){15}}
     \put(34.75,1){\line(0,1){15}}
     \put(19.75,1){\line(1,1){15}}
     \end{picture}
    }
\put(46,47){\vector(1,0){10}} \put(82,38){\vector(0,-1){10}}
\put(56,17){\vector(-1,0){10}} \put(23,29){\vector(0,1){9}}

\put(17,36){$p$} \put(17,60){$q$}

\end{picture}
}


\put(10,88){\begin{picture}(97,56) \put(-5,39){
     \begin{picture}(47,16)
     \put(6,1){\circle*{1}}
     \put(6,16){\circle*{1}}
     \put(21,1){\circle*{1}}
     \put(21,16){\circle*{1}}
     \put(26,1){\circle*{1}}
     \put(26,16){\circle*{1}}
     \put(41,1){\circle*{1}}
     \put(41,16){\circle*{1}}
     \put(6,1){\line(1,0){15}}
     \put(6,16){\line(1,0){15}}
     \put(6,1){\line(0,1){15}}
     \put(21,1){\line(0,1){15}}
     \put(6,16){\line(1,-1){15}}
     \put(26,1){\line(1,0){15}}
     \put(26,1){\line(0,1){15}}
     \put(26,16){\line(1,0){15}}
     \put(41,1){\line(0,1){15}}
     \put(26,16){\line(1,-1){15}}
     \end{picture}
    }
\put(50,39){
     \begin{picture}(47,16)

     \put(6,1){\circle*{1}}
     \put(6,16){\circle*{1}}
     \put(21,1){\circle*{1}}
     \put(21,16){\circle*{1}}
     \put(26,1){\circle*{1}}
     \put(26,16){\circle*{1}}
     \put(41,1){\circle*{1}}
     \put(41,16){\circle*{1}}
     \put(6,1){\line(1,0){15}}
     \put(6,16){\line(1,0){15}}
     \put(6,1){\line(0,1){15}}
     \put(21,1){\line(0,1){15}}
     \put(6,1){\line(1,1){15}}
     \put(26,1){\line(1,0){15}}
     \put(26,1){\line(0,1){15}}
     \put(26,16){\line(1,0){15}}
     \put(41,1){\line(0,1){15}}
     \put(26,16){\line(1,-1){15}}
     \end{picture}
    }
\put(50,9){
     \begin{picture}(47,16)

     \put(6,1){\circle*{1}}
     \put(6,16){\circle*{1}}
     \put(21,1){\circle*{1}}
     \put(21,16){\circle*{1}}
     \put(26,1){\circle*{1}}
     \put(26,16){\circle*{1}}
     \put(41,1){\circle*{1}}
     \put(41,16){\circle*{1}}
     \put(6,1){\line(1,0){15}}
     \put(6,16){\line(1,0){15}}
     \put(6,1){\line(0,1){15}}
     \put(21,1){\line(0,1){15}}
     \put(6,1){\line(1,1){15}}
     \put(26,1){\line(1,0){15}}
     \put(26,1){\line(0,1){15}}
     \put(26,16){\line(1,0){15}}
     \put(41,1){\line(0,1){15}}
     \put(26,1){\line(1,1){15}}
     \end{picture}
    }
\put(-5,9){
     \begin{picture}(47,16)

     \put(6,1){\circle*{1}}
     \put(6,16){\circle*{1}}
     \put(21,1){\circle*{1}}
     \put(21,16){\circle*{1}}
     \put(26,1){\circle*{1}}
     \put(26,16){\circle*{1}}
     \put(41,1){\circle*{1}}
     \put(41,16){\circle*{1}}
     \put(6,1){\line(1,0){15}}
     \put(6,16){\line(1,0){15}}
     \put(6,1){\line(0,1){15}}
     \put(21,1){\line(0,1){15}}
     \put(6,16){\line(1,-1){15}}
     \put(26,1){\line(1,0){15}}
     \put(26,1){\line(0,1){15}}
     \put(26,16){\line(1,0){15}}
     \put(41,1){\line(0,1){15}}
     \put(26,1){\line(1,1){15}}
     \end{picture}
    }
\put(43,47){\vector(1,0){10}} \put(75,38){\vector(0,-1){10}}
\put(53,17){\vector(-1,0){10}} \put(20,28){\vector(0,1){10}}

\end{picture}
}


\put(126,83){\begin{picture}(102,64)

\put(0,39){
     \begin{picture}(47,16)
     \put(1,13){\circle*{1}}
     \put(25,13){\circle*{1}}
     \put(13,1){\circle*{1}}
     \put(13,25){\circle*{1}}
     \put(37,1){\circle*{1}}
     \put(37,25){\circle*{1}}
     \put(49,13){\circle*{1}}
     \put(1,13){\line(1,1){12}}
     \put(1,13){\line(1,-1){12}}
     \put(13,1){\line(1,1){24}}
     \put(13,25){\line(1,-1){24}}
     \put(49,13){\line(-1,-1){12}}
     \put(49,13){\line(-1,1){12}}
     \put(13,1){\line(0,1){24}}
     \put(37,1){\line(0,1){24}}

     \qbezier(27,15.5)(25,18)(22.5,15.5)
     \put(23,21){$p$}

     \qbezier(27,10.5)(25,8)(22.5,10.5)
     \put(23,3){$q$}

     \end{picture}
    }
\put(62,39){
     \begin{picture}(47,16)

     \put(1,13){\circle*{1}}
     \put(25,13){\circle*{1}}
     \put(13,1){\circle*{1}}
     \put(13,25){\circle*{1}}
     \put(37,1){\circle*{1}}
     \put(37,25){\circle*{1}}
     \put(49,13){\circle*{1}}
     \put(1,13){\line(1,1){12}}
     \put(1,13){\line(1,-1){12}}
     \put(13,1){\line(1,1){24}}
     \put(13,25){\line(1,-1){24}}
     \put(49,13){\line(-1,-1){12}}
     \put(49,13){\line(-1,1){12}}
     \put(1,13){\line(1,0){24}}
     \put(37,1){\line(0,1){24}}
     \end{picture}
    }
\put(62,9){
     \begin{picture}(47,16)

     \put(1,13){\circle*{1}}
     \put(25,13){\circle*{1}}
     \put(13,1){\circle*{1}}
     \put(13,25){\circle*{1}}
     \put(37,1){\circle*{1}}
     \put(37,25){\circle*{1}}
     \put(49,13){\circle*{1}}
     \put(1,13){\line(1,1){12}}
     \put(1,13){\line(1,-1){12}}
     \put(13,1){\line(1,1){24}}
     \put(13,25){\line(1,-1){24}}
     \put(49,13){\line(-1,-1){12}}
     \put(49,13){\line(-1,1){12}}
     \put(1,13){\line(1,0){48}}

     \end{picture}
    }
\put(0,9){
     \begin{picture}(47,16)

     \put(1,13){\circle*{1}}
     \put(25,13){\circle*{1}}
     \put(13,1){\circle*{1}}
     \put(13,25){\circle*{1}}
     \put(37,1){\circle*{1}}
     \put(37,25){\circle*{1}}
     \put(49,13){\circle*{1}}
     \put(1,13){\line(1,1){12}}
     \put(1,13){\line(1,-1){12}}
     \put(13,1){\line(1,1){24}}
     \put(13,25){\line(1,-1){24}}
     \put(49,13){\line(-1,-1){12}}
     \put(49,13){\line(-1,1){12}}
     \put(13,1){\line(0,1){24}}
     \put(25,13){\line(1,0){24}}
     \end{picture}
    }
\put(53.5,51.5){\vector(1,0){9}} \put(89.5,42){\vector(0,-1){10}}
\put(62.5,21.5){\vector(-1,0){9}} \put(27.5,28){\vector(0,1){9}}
\end{picture}
}


\put(69,0){\begin{picture}(102,64)

\put(0,46){
     \begin{picture}(47,16)
     \put(1,1){\circle*{1}}
     \put(1,21){\circle*{1}}
     \put(21,1){\circle*{1}}
     \put(21,21){\circle*{1}}
     \put(41,1){\circle*{1}}
     \put(41,21){\circle*{1}}

     \put(1,1){\line(1,0){40}}
     \put(1,21){\line(1,0){40}}
     \put(1,1){\line(0,1){20}}
     \put(21,1){\line(0,1){20}}
     \put(41,1){\line(0,1){20}}

     \put(1,1){\line(1,1){20}}
     \put(21,1){\line(1,1){20}}
     \put(21,21){\oval(5,5)[t]}  \put(21,1){\oval(5,5)[b]}

     \end{picture}
    }
\put(55,46){
     \begin{picture}(47,16)

     \put(1,1){\circle*{1}}
     \put(1,21){\circle*{1}}
     \put(21,1){\circle*{1}}
     \put(21,21){\circle*{1}}
     \put(41,1){\circle*{1}}
     \put(41,21){\circle*{1}}

     \put(1,1){\line(1,0){40}}
     \put(1,21){\line(1,0){40}}
     \put(1,1){\line(0,1){20}}
     \put(21,1){\line(0,1){20}}
     \put(41,1){\line(0,1){20}}

     \put(1,21){\line(1,-1){20}}
     \put(21,1){\line(1,1){20}}
     \end{picture}
    }
\put(55,9){
     \begin{picture}(47,16)

     \put(1,1){\circle*{1}}
     \put(1,21){\circle*{1}}
     \put(21,1){\circle*{1}}
     \put(21,21){\circle*{1}}
     \put(41,1){\circle*{1}}
     \put(41,21){\circle*{1}}

     \put(1,1){\line(1,0){40}}
     \put(1,21){\line(1,0){40}}
     \put(1,1){\line(0,1){20}}
     \put(21,1){\line(0,1){20}}
     \put(41,1){\line(0,1){20}}

     \put(1,21){\line(1,-1){20}}
     \put(21,21){\line(1,-1){20}}
     \end{picture}
    }
\put(0,9){
     \begin{picture}(47,16)

     \put(1,1){\circle*{1}}
     \put(1,21){\circle*{1}}
     \put(21,1){\circle*{1}}
     \put(21,21){\circle*{2}}
     \put(41,1){\circle*{1}}
     \put(41,21){\circle*{1}}

     \put(1,1){\line(1,0){40}}
     \put(1,21){\line(1,0){40}}
     \put(1,1){\line(0,1){20}}
     \put(21,1){\line(0,1){20}}
     \put(41,1){\line(0,1){20}}

     \put(1,1){\line(1,1){20}}
     \put(21,21){\line(1,-1){20}}
     \end{picture}
    }
\put(46,57){\vector(1,0){10}} \put(79,43){\vector(0,-1){10}}
\put(56,19){\vector(-1,0){10}} \put(24,33){\vector(0,1){10}}

\put(19,40){$p$} \put(22,72){$q$}

\end{picture}
}

\it

\put(54,247){c}

\put(182,330){b}

\put(54,330){a}

\put(182,247){d}

\put(54,164){e}

\put(182,164){f}

\put(54,81){g}

\put(182,81){h}

\put(118,-2){i}

\end{picture}
}

\caption{Elementary cycles of the first type}\label{fig_comm}
\end{figure}
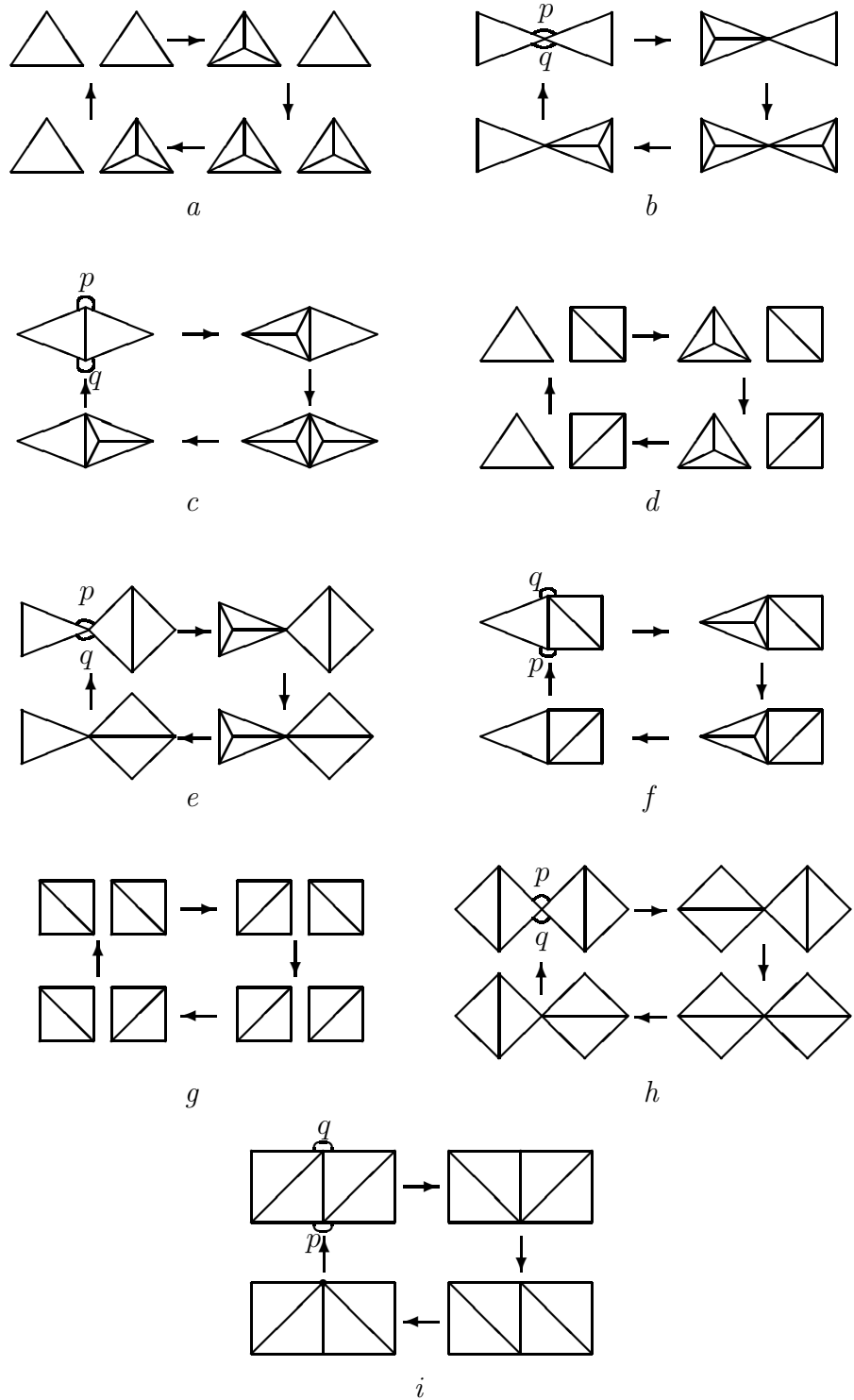

\begin{figure}[t]

{\unitlength=0.5mm
\begin{picture}(240,200)

\thicklines

\put(10,112){\begin{picture}(72,75)

\put(20,49){
     \begin{picture}(32,22)
     \put(1,1){\circle*{1}}
     \put(16,8.5){\circle*{1}}
     \put(16,21){\circle*{1}}
     \put(31,1){\circle*{1}}
     \put(1,1){\line(1,0){30}}
     \put(1,1){\line(3,4){15}}
     \put(1,1){\line(2,1){15}}
     \put(16,8.5){\line(0,1){12.5}}
     \put(16,8.5){\line(2,-1){15}}
     \put(16,21){\line(3,-4){15}}

     \put(1,1){\oval(3.5,3.5)[b]}
     \put(1,1){\oval(3.5,3.5)[tl]}
     \qbezier(1.5,2.4)(1.65625,2.7)(1,2.7)

     \put(31,1){\oval(3.5,3.5)[b]}
     \put(31,1){\oval(3.5,3.5)[tr]}
     \qbezier(29.95,2.4)(30.34375,2.7)(31,2.7)

     \put(15.7,21){\oval(3.5,3.5)[t]}
     \qbezier(14.5,19.6)(13.9,20)(13.9,21)
     \qbezier(16.85,19.6)(17.4,20)(17.4,21)
     \end{picture}
    }
\put(50,9){
     \begin{picture}(32,28)
     \put(1,1){\circle*{1}}
     \put(11,11){\circle*{1}}
     \put(16,26){\circle*{1}}
     \put(21,11){\circle*{1}}
     \put(31,1){\circle*{1}}
     \put(1,1){\line(1,0){30}}
     \put(1,1){\line(3,5){15}}
     \put(1,1){\line(1,1){10}}
     \put(11,11){\line(1,0){10}}
     \put(11,11){\line(1,3){5}}
     \put(16,26){\line(3,-5){15}}
     \put(16,26){\line(1,-3){5}}
     \put(21,11){\line(1,-1){10}}
     \put(11,11){\line(2,-1){20}}
     \end{picture}
    }
\put(-10,9){
     \begin{picture}(32,28)
     \put(1,1){\circle*{1}}
     \put(11,11){\circle*{1}}
     \put(16,26){\circle*{1}}
     \put(21,11){\circle*{1}}
     \put(31,1){\circle*{1}}
     \put(1,1){\line(1,0){30}}
     \put(1,1){\line(3,5){15}}
     \put(1,1){\line(1,1){10}}
     \put(11,11){\line(1,0){10}}
     \put(11,11){\line(1,3){5}}
     \put(16,26){\line(3,-5){15}}
     \put(16,26){\line(1,-3){5}}
     \put(21,11){\line(1,-1){10}}
     \put(1,1){\line(2,1){20}}
     \end{picture}
    }

\put(53,20){\vector(-1,0){30}} \put(52,45){\vector(3,-4){10}}
\put(15,32){\vector(3,4){10}} \put(16,46){$p$} \put(37,75){$q$}
\put(58,46){$r$}


\end{picture}
}


\put(130,112){\begin{picture}(101,80)

\put(4,49){
     \begin{picture}(22,22)
     \put(1,1){\circle*{1}}
     \put(21,1){\circle*{1}}
     \put(1,21){\circle*{1}}
     \put(21,21){\circle*{1}}
     \put(1,1){\line(1,0){20}}
     \put(1,1){\line(0,1){20}}
     \put(1,21){\line(1,0){20}}
     \put(21,1){\line(0,1){20}}
     \put(1,1){\line(1,1){20}}
     \put(1,1){\oval(3.5,3.5)[b]}
       \put(1,1){\oval(3.5,3.5)[tl]}
         \put(1,21){\oval(3.5,3.5)[t]}
           \put(1,21){\oval(3.5,3.5)[bl]}
             \put(21,21){\oval(3.5,3.5)[t]}
               \put(21,21){\oval(3.5,3.5)[rb]}
                 \put(21,1){\oval(3.5,3.5)[b]}
                   \put(21,1){\oval(3.5,3.5)[rt]}
     \end{picture}
    }
\put(44,49){
     \begin{picture}(22,22)
     \put(1,1){\circle*{1}}
     \put(21,1){\circle*{1}}
     \put(1,21){\circle*{1}}
     \put(21,21){\circle*{1}}
     \put(11,6){\circle*{1}}
     \put(1,1){\line(1,0){20}}
     \put(1,1){\line(0,1){20}}
     \put(1,21){\line(1,0){20}}
     \put(21,1){\line(0,1){20}}
     \put(1,1){\line(1,1){20}}
     \put(1,1){\line(2,1){10}}
     \put(11,6){\line(2,3){10}}
     \put(11,6){\line(2,-1){10}}
     \end{picture}
    }
\put(79,29){
     \begin{picture}(22,22)
     \put(1,1){\circle*{1}}
     \put(21,1){\circle*{1}}
     \put(1,21){\circle*{1}}
     \put(21,21){\circle*{1}}
     \put(11,6){\circle*{1}}
     \put(1,1){\line(1,0){20}}
     \put(1,1){\line(0,1){20}}
     \put(1,21){\line(1,0){20}}
     \put(21,1){\line(0,1){20}}
     \put(1,21){\line(2,-3){10}}
     \put(1,1){\line(2,1){10}}
     \put(11,6){\line(2,3){10}}
     \put(11,6){\line(2,-1){10}}
     \end{picture}
    }
\put(44,9){
     \begin{picture}(22,22)
     \put(1,1){\circle*{1}}
     \put(21,1){\circle*{1}}
     \put(1,21){\circle*{1}}
     \put(21,21){\circle*{1}}
     \put(11,6){\circle*{1}}
     \put(1,1){\line(1,0){20}}
     \put(1,1){\line(0,1){20}}
     \put(1,21){\line(1,0){20}}
     \put(21,1){\line(0,1){20}}
     \put(1,21){\line(1,-1){20}}
     \put(1,1){\line(2,1){10}}
     \put(1,21){\line(2,-3){10}}
     \put(11,6){\line(2,-1){10}}
     \end{picture}
    }

\put(4,9){
     \begin{picture}(22,22)
     \put(1,1){\circle*{1}}
     \put(21,1){\circle*{1}}
     \put(1,21){\circle*{1}}
     \put(21,21){\circle*{1}}
     \put(1,1){\line(1,0){20}}
     \put(1,1){\line(0,1){20}}
     \put(1,21){\line(1,0){20}}
     \put(21,1){\line(0,1){20}}
     \put(1,21){\line(1,-1){20}}
     \end{picture}
    }

\put(32,60){\vector(1,0){10}} \put(70,52){\vector(2,-1){10}}
\put(42,20){\vector(-1,0){10}} \put(17,35){\vector(0,1){10}}
\put(80,32){\vector(-2,-1){10}} \put(0.5,45){$p$}
\put(0.5,73){$q$} \put(30.5,73){$r$} \put(30.5,45){$k$}


\end{picture}
}


\put(65,10){\begin{picture}(101,75)

\put(4,49){
     \begin{picture}(22,22)
     \put(6,1){\circle*{1}}
     \put(1,11){\circle*{1}}
     \put(11,21){\circle*{1}}
     \put(16,1){\circle*{1}}
     \put(21,11){\circle*{1}}
     \put(1,11){\line(1,1){10}}
     \put(1,11){\line(1,-2){5}}
     \put(6,1){\line(1,0){10}}
     \put(11,21){\line(1,-1){10}}
     \put(16,1){\line(1,2){5}}
     \put(6,1){\line(1,4){5}}
     \put(6,1){\line(3,2){15}}

     \put(5.7,1){\oval(3.5,3.5)[b]}
     \put(15.7,1){\oval(3.5,3.5)[b]}
     \put(1,11){\oval(3.5,3.5)[l]}
     \put(21,11){\oval(3.5,3.5)[r]}
     \put(10.7,21){\oval(3.5,3.5)[t]}
     \qbezier(3.9,1)(3.9,2.35)(5,2.6)
     \qbezier(17.5,1)(17.5,2.35)(16.6,2.5)
     \qbezier(1,9.25)(1.2,9.25)(1.4,9.6)
     \qbezier(21,9.25)(20.6,9.25)(20.2,9.4)
     \qbezier(1,12.65)(1.5,12.65)(2,12.3)
     \qbezier(21,12.75)(20.5,12.75)(19.7,12.3)
     \qbezier(8.9,21)(8.9,20.1)(9.4,19.7)
     \qbezier(12.4,21)(12.4,20.1)(12,19.7)
     \end{picture}
    }
\put(44,49){
     \begin{picture}(22,22)
     \put(6,1){\circle*{1}}
     \put(1,11){\circle*{1}}
     \put(11,21){\circle*{1}}
     \put(16,1){\circle*{1}}
     \put(21,11){\circle*{1}}
     \put(1,11){\line(1,1){10}}
     \put(1,11){\line(1,-2){5}}
     \put(6,1){\line(1,0){10}}
     \put(11,21){\line(1,-1){10}}
     \put(16,1){\line(1,2){5}}
     \put(1,11){\line(1,0){20}}
     \put(6,1){\line(3,2){15}}
     \end{picture}
    }
\put(79,29){
     \begin{picture}(22,22)
     \put(6,1){\circle*{1}}
     \put(1,11){\circle*{1}}
     \put(11,21){\circle*{1}}
     \put(16,1){\circle*{1}}
     \put(21,11){\circle*{1}}
     \put(1,11){\line(1,1){10}}
     \put(1,11){\line(1,-2){5}}
     \put(6,1){\line(1,0){10}}
     \put(11,21){\line(1,-1){10}}
     \put(16,1){\line(1,2){5}}
     \put(1,11){\line(1,0){20}}
     \put(1,11){\line(3,-2){15}}
     \end{picture}
    }
\put(44,9){
     \begin{picture}(22,22)
     \put(6,1){\circle*{1}}
     \put(1,11){\circle*{1}}
     \put(11,21){\circle*{1}}
     \put(16,1){\circle*{1}}
     \put(21,11){\circle*{1}}
     \put(1,11){\line(1,1){10}}
     \put(1,11){\line(1,-2){5}}
     \put(6,1){\line(1,0){10}}
     \put(11,21){\line(1,-1){10}}
     \put(16,1){\line(1,2){5}}
     \put(11,21){\line(1,-4){5}}
     \put(1,11){\line(3,-2){15}}
     \end{picture}
    }

\put(4,9){
     \begin{picture}(22,22)
     \put(6,1){\circle*{1}}
     \put(1,11){\circle*{1}}
     \put(11,21){\circle*{1}}
     \put(16,1){\circle*{1}}
     \put(21,11){\circle*{1}}
     \put(1,11){\line(1,1){10}}
     \put(1,11){\line(1,-2){5}}
     \put(6,1){\line(1,0){10}}
     \put(11,21){\line(1,-1){10}}
     \put(16,1){\line(1,2){5}}
     \put(6,1){\line(1,4){5}}
     \put(11,21){\line(1,-4){5}}
     \end{picture}
    }

\put(33,60){\vector(1,0){10}} \put(70,52){\vector(2,-1){10}}
\put(43,20){\vector(-1,0){10}} \put(17.5,35){\vector(0,1){10}}
\put(80,32){\vector(-2,-1){10}}

\put(5,45){$p$}

\put(1,63){$q$}

\put(16,74){$r$}

\put(30,63){$k$}

\put(27,45){$l$}


\end{picture}
}

\put(119,6){\textit{c}}

\put(182,106){\textit{b}}

\put(47,106){\textit{a}}
\end{picture}}

\caption{Elementary cycles of the second type}\label{fig_dop}
\end{figure}
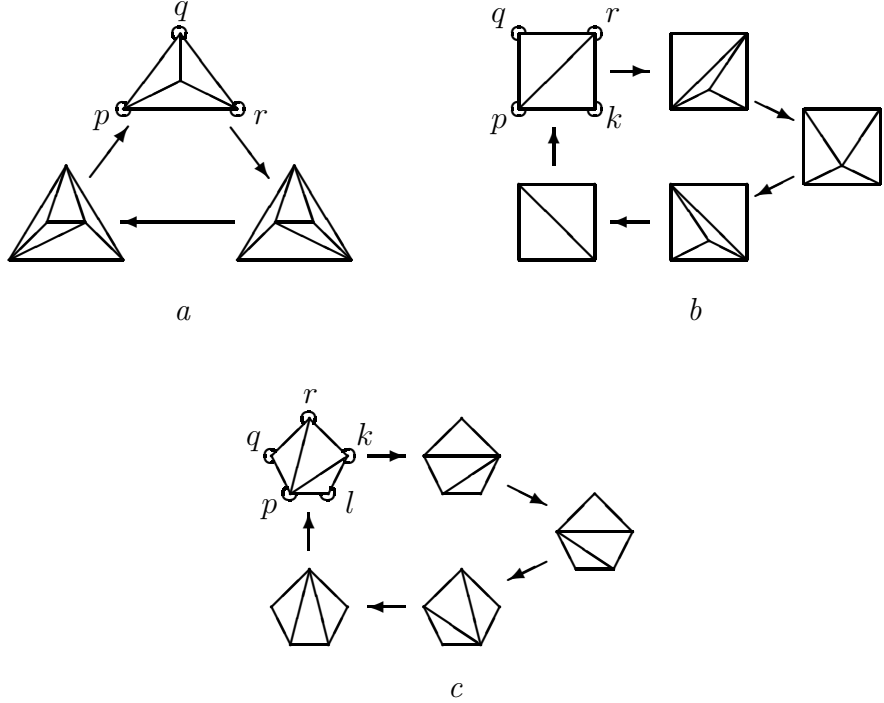

\begin{propos}\label{propos_cycly}
Any cycle of the graph~$\Gamma_2$ can be decomposed into a linear combination of elementary cycles with integral coefficients.
\end{propos}
\begin{proof}[Sketch of proof]
It will be convenient to us to give the following auxiliary definition.
Let $L$ be a two-dimensional combinatorial sphere with $k$ vertices. We shall say that its \textit{complexity} $a(L)$ is equal to~$k$ if $L$ contains at least one vertex of degree~$3$, is equal to~$k+\frac13$ if~$L$ contains no vertices of degree~$3$ and contains at least one vertex of degree~$4$, and is equal to~$k+\frac23$ if $L$ contains no vertices of degrees~$3$ and~$4$. The Euler formula implies that, in the latter case, $L$ contains at least $12$ vertices of degree~$5$. Thus the complexity of a combinatorial sphere is an element of the set~$\frac13\Z_{\ge 0}$. Let $\beta$ be a bistellar move transforming a two-dimensional combinatorial sphere~$L_1$ to a two-dimensional combinatorial sphere~$L_2$. We put
$a(\beta)=\max\bigl(a(L_1),a(L_2)\bigr)$ if $a(L_1)\ne a(L_2)$ and
$a(\beta)=a(L_1)+\frac16$ if $a(L_1)= a(L_2)$. Then the \textit{complexity}~$a(\beta)$ of a bistellar move~$\beta$ is an element of the set~$\frac16\Z_{\ge 0}$. By~$\Gamma_2^a$, $a\in\frac16\Z_{\ge 0}$, we denote the subgraph of~$\Gamma_2$ consisting of all vertices and edges of complexity not greater than~$a$.

Obviously, to prove the proposition we suffice to prove that any relative cellular cycle of the pair~$\bigl(\Gamma^{a\vphantom{\frac16}}_2,\Gamma^{a-\frac16}_2\bigr)$ can be presented as a sum of a linear combination of elementary cycles and a cellular chain of the graph~$\Gamma^{a-\frac16}_2$. This assertion should be proved separately for each of the six cases $a\in\Z_{\ge 0}+\frac{b}6$, $b=0,1,2,3,4,5$. Below we shall consider in details the cases~$b=1$ and~$b=4$. The case~$b=0$ is trivial; the case $b=2$ is almost similar to the case~$b=4$; the cases~$b=3$ and~$b=5$ are almost similar to the case~$b=1$. (The details of the proof in different cases slightly differ.)

{\sl Case $a=k+\frac16$, $k\in\Z$.} Notice that the sets of vertices of the graphs~$\Gamma^{k+\frac16}_2$ and~$\Gamma^{k\vphantom{\frac16}}_2$ coincide. Hence the group of relative cycles of the pair
$\bigl(\Gamma^{k+\frac16}_2,\Gamma^{k\vphantom{\frac16}}_2\bigr)$ is generated by edges~$\{\beta\}$ such that~$\beta$ is an essential bistellar move transforming~$L_1$ to~$L_2$, where either of the combinatorial spheres~$L_1$ and~$L_2$ has~$k$ vertices and contains a vertex of degree~$3$. Then
$\beta=\beta_{L_1,\sigma}$, where $\dim\sigma=1$. We shall consider two subcases.

1. There is a vertex $v\in V(L_1)=V(L_2)$ such that the degree of~$v$ in either of the triangulations~$L_1$ and~$L_2$ is equal to~$3$.
Then $v\notin U(\beta)$. Hence the cycle~$\gamma(L_1,\sigma,v)$ is well defined. Now we suffice to notice that the support of the chain $\{\beta\}-\gamma(L_1,\sigma,v)$ is contained in the graph~$\Gamma^k_2$.

2. There is not a vertex $v\in V(L_1)=V(L_2)$ such that the degree of~$v$ in either of the triangulations~$L_1$ and~$L_2$ is equal to~$3$.
Then there are two vertices $v_1,v_2\in U(\beta)$ such that the degrees of~$v_1$ in the triangulations~$L_1$ and~$L_2$ are equal to~$3$ and~$4$ respectively and the degrees of~$v_2$ in the triangulations~$L_1$ and~$L_2$ are equal to~$4$ and~$3$ respectively. Then the bistellar move~$\beta$ is arranged either as the bistellar move shown in the bottom of Fig.~\ref{fig_dop},~\textit{a} or as the bistellar move inverse to the bistellar move shown in the bottom of Fig.~\ref{fig_dop},~\textit{a}. Therefore, either adding to~$\{\beta\}$ or subtracting from~$\{\beta\}$ the cycle shown in this figure, we obtain a chain with support in~$\Gamma^k_2$.

{\sl Case $a=k+\frac23$, $k\in\Z$.} The group of relative cycles of~$\bigl(\Gamma^{k+\frac23}_2,\Gamma^{k+\frac12}_2\bigr)$ is generated by elements $\{\beta_1\}-\{\beta_2\}$ such that $\beta_1$ and $\beta_2$ are bistellar moves transforming~$L$ to~$L_1$ and~$L_2$ respectively, where $L$ is an oriented combinatorial sphere with
$k$ vertices without vertices of degrees~$3$ and~$4$, $L_1$ and
$L_2$ are oriented combinatorial spheres either of which has $k$ vertices and contains a vertex of degree~$4$. Then, for $i=1,2$, we can see that $\beta_i=\beta_{L,\sigma_i}$, where $\sigma_i$ is an edge entering a vertex of degree~$5$. The triangulation~$L$ contains at least $12$ vertices of degree~$5$. Hence the triangulation~$L$ contains rather many (at least~$36$) edges~$\sigma$ such that $\beta_{L,\sigma}$ is a bistellar move of complexity~$k+\frac23$. By an easy analysis of cases one can show that the group of relative cycles of the pair~$(\Gamma^{k+\frac23}_2,\Gamma^{k+\frac12}_2)$ is generated by those  differences~$\{\beta_1\}-\{\beta_2\}$ for which the cycle~$\gamma(L,\sigma_1,\sigma_2)$ is well defined.
Then the support of the chain
$$
\{\beta_1\}-\{\beta_2\}-\gamma(L,\sigma_1,\sigma_2)
$$
is contained in the graph~$\Gamma_2^{k+\frac12}$.
\end{proof}
\begin{remark}
Notice that the above proof of Proposition~\ref{propos_cycly} actually gives us an explicit algorithm for decomposing a given cycle of the graph~$\Gamma_2$ into a linear combination of elementary cycles. This algorithm is important for finding explicit local formula for the first Pontryagin class. Indeed, it can be easily proved that elementary cycles of the second type can be excluded because they can be presented as linear combinations of elementary cycles of the first type.
Nevertheless, it is convenient to use elementary cycles of both types because an algorithm for decomposing a given cycle into a linear combination of elementary cycles of the first type only is more complicated. Proposition~\ref{propos_cycly} has been initially proved by the author in~\cite{Gai04} using another method based on realisation of two-dimensional combinatorial spheres as the boundaries of simplicial convex polytopes in~$\R^3$ and the Steinitz theorem claiming that, for any two such realisations, the first one can be deformed to the second one in the class of simplicial convex polytopes. This proof also yields an explicit algorithm for decomposing a given cycle. However, this algorithm is much more complicated than the algorithm described above.
\end{remark}
\begin{remark}
Attach to the graph~$\Gamma_2$ a two-dimensional cell along every elementary cycle. Denote by~$X$ the obtained two-dimensional complex. Proposition~\ref{propos_cycly} claims that $H_1(X;\Z)=0$. It can be easily shown that the above proof of Proposition~\ref{propos_cycly} actually provides that the complex~$X$ is simply connected.
\end{remark}

\section{Explicit formula for the first Pontryagin class}
\label{section_explicit}

As we mentioned in the end of section~\ref{section_locform}, elements of the group
$$
B=\ker\bigl[\delta:\T^4(\Q)\to
\T^5(\Q)\bigr]=\ker\bigl[\delta:C^0(\Gamma_3;\CQ)\to
C^0(\Gamma_4;\CQ)\bigr]
$$
are exactly local formulae for the first rational Pontryagin class multiplied by some rational number. Therefore we need to describe explicitly the group~$B$. The equality $d=\delta s-s\delta$ implies that $d|_B=\delta s|_B$ and
$d|_{C^0_{\Z_2}(\Gamma_2;\CQ)}=-s\delta|_{C^0_{\Z_2}(\Gamma_2;\CQ)}$.
But $d$ is a monomorphism. Hence the homomorphism~$s|_B$ is a monomorphism and the homomorphism~$s|_{\delta\bigl(C^0_{\Z_2}(\Gamma_2;\CQ)\bigr)}$ is an isomorphism onto the group~$d\bigl(C^0_{\Z_2}(\Gamma_2;\CQ)\bigr)$. Besides, the group~$s(B)$ is contained in the subgroup~$A\subset C^1_{\Z_2}(\Gamma_2;\CQ)$ consisting of all cocycles whose homology classes belong to the group
$$
N=\ker\bigl[\delta^*: H^1_{\Z_2}(\Gamma_2;\CQ)\to
H^1_{\Z_2}(\Gamma_3;\CQ)\bigr],
$$
where $\delta^*$ is the mapping in cohomology induced by the chain mapping~$\delta$ of cochain complexes.

The group~$N$ has been computed by the author in~\cite{Gai04}. It is a one-dimensional $\Q$-module generated by the cohomology class~$c$ whose values on elementary cycles are given in Table~\ref{table}. The right column of this table contains the values of~$c$ on the cycles shown in the figures whose numbers are written in the left column. These values depend on the numbers of triangles containing the vertices shown in the corresponding figures and situated inside the angles marked by arcs. The number of simplices inside the angle is indicated by the letter written near the corresponding arc. The functions~$\rho$ and~$\omega$ are given by
\begin{gather*}
\rho(p,q)=\frac{q-p}{(p+q+2)(p+q+3)(p+q+4)};\\
\omega(p)=\frac{1}{(p+2)(p+3)}
\end{gather*}
Proposition~\ref{propos_cycly} easily implies that the class~$c$ is solely determined by its values on elementary cycles.

\begin{table}
\caption{Values of~$c$ on elementary cycles}\label{table}
\begin{tabular}{|l|c|}
\hline \ref{fig_comm}, \textit{a}, \textit{d}, \textit{g} & $0\vphantom{\underline{\frac12}^0}$\\
\hline \ref{fig_comm}, \textit{b}, \textit{e}, \textit{h} & $\rho(p,q)\vphantom{\underline{\frac12}^0}$\\
\hline \ref{fig_comm}, \textit{c}, \textit{i} & $\rho(0,q)-\rho(0,p)\vphantom{\underline{\frac12}^0}$\\
\hline \ref{fig_comm}, \textit{f} & $\rho(0,q)+\rho(0,p)\vphantom{\underline{\frac12}^0}$\\
\hline \ref{fig_dop}, \textit{a} & $\omega(p)-\omega(q)+\omega(r)-\frac1{12}\vphantom{\underline{\frac12}^0}$\\
\hline \ref{fig_dop}, \textit{b} & $\omega(p)-\omega(q)-\omega(r)+\omega(k)\vphantom{\underline{\frac12}^0}$\\
\hline \ref{fig_dop}, \textit{c} & $\omega(p)+\omega(q)+\omega(r)+\omega(k)+\omega(l)-\frac1{12}\vphantom{\underline{\frac12}^0}$\\
\hline
\end{tabular}
\end{table}

\begin{remark}
Notice that a function very similar to the function~$\rho$ has appeared in the work~\cite{Kaz98} by Kazarian in the formula for the Euler class of $S^1$-bundle from degeneracies of the restriction of the Morse function on the total space to fibres.
\end{remark}

Since the dimension of the $\Q$-module~$N$ is equal to~$1$, it follows that the $\Q$-module $d\bigl(C^0_{\Z_2}(\Gamma_2;\CQ)\bigr)$ has codimension~$1$ in the $\Q$-module~$A$. On the other hand, Theorem~\ref{theorem_isom} implies that the $\Q$-module~$\delta
\bigl(C^0_{\Z_2}(\Gamma_2;\CQ)\bigr)$ has codimension~$1$ in the 
$\Q$-module~$B$. Hence the monomorphism $s|_B:B\to A$ is an isomorphism. Besides, $(s|_B)^{-1}=d^{-1}\delta|_A$. Therefore, for any cocycle $\hc\in C^1_{\Z_2}(\Gamma_2;\CQ)$ representing the cohomology class~$c$, the function $f=d^{-1}\delta(\hc)\in \T^4(\Q)$ is a local formula for the class~$\lambda p_1$, where $\lambda$ is a certain rational constant. To compute~$\lambda$, we need to compute the sum $\sum_{v\in V(K)}f\bigl(\langle\Lk v\rangle\bigr)$ for some oriented four-dimensional combinatorial manifold~$K$ with non-zero first Pontryagin number and to compare the value computed with the first Pontryagin number of~$K$. A direct computation for the $9$-vertex triangulation of~$\mathbb{C}\mathrm{P}^2$ constructed by K\"uhnel and Banchoff~\cite{KB83} yields~$\lambda=1$. Thus local formulae for the first Pontryagin class are in one-to-one correspondence with cocycles~$\hc$ representing the cohomology class~$c$. This correspondence is provided by the two mutually inverse mappings~$s|_B$ and~$d^{-1}\delta|_A$.

To describe explicitly a concrete local formula~$f$ for the first rational Pontryagin class we need to choose a cocycle~$h$ representing the cohomology class~$c$. This can be done, for example, in the following way. For every vertex~$\{ L\}$ of the graph~$\Gamma_2$, we shall point out explicitly a chain $\xi_{\{ L\}}\in
C_1(\Gamma_2;\Z)$ such that $\partial \xi_{\{ L\}}=\{ L\}-
\{\partial\Delta^3\}$. We put, $\xi_{\{\partial\Delta^3\}}=0$.
We shall give a recurrent formula for the computation of the chain $\xi_{\{ L\}}$ from the chains $\xi$ for vertices of less complexity. Let $\beta_1,\beta_2,\ldots,\beta_{r}$ be all bistellar moves decreasing the complexity of the combinatorial sphere~$L$. (It is easy to prove that such bistellar moves always exist.) We put, $L_i=\beta_i(L)$, $i=1,2,\ldots,r$. The chain~$\xi_{\{L\}}$ is determined by the formula 
$$
\xi_{\{L\}}=\frac1r\sum_{j=1}^r\bigl(\xi_{\{L_j\}}-\{\beta_j\}\bigr).
$$
Now the cocylce~$h$ is defined by
$$
h(\{\beta\})=\left\langle
c,\{\beta\}+\xi_{\{L\}}-\xi_{\{\beta(L)\}}\right\rangle,
$$
where $L$ is the initial combinatorial sphere of the bistellar move~$\beta$.

The procedure for computing the value~$f\bigl(\langle L\rangle\bigr)$
is as follows.

1. We find a sequence of bistellar moves
\begin{equation}\label{eq_bistellar}
\partial\Delta^4=L_1\stackrel{\beta_1}{\rightsquigarrow}L_2
\stackrel{\beta_2}{\rightsquigarrow}\ldots
\stackrel{\beta_{k-1}}{\rightsquigarrow}L_k
\stackrel{\beta_k}{\rightsquigarrow}L_{k+1}\cong L
\end{equation}
transforming the boundary of a $4$-dimensional simplex to~$L$.

2. We put,
$$
\eta=\sum_{j=1}^k\sum_{v\in
U(\beta_j)}\left\{(\beta_j)_v\right\}\in C_1(\Gamma_2;\Z).
$$

3. We compute recurrently the chains~$\xi_{\{\Lk v\}}$ for all~$v\in
V(L)$. Then the chain $\zeta=\eta-\sum_{v\in V(L)}\xi_{\{\Lk v\}}$ is a cycle.

4. We decompose the cycle~$\zeta$ as a linear combination 
$\sum_{i=1}^ln_i\gamma_i$, where $n_i\in\Z$ and $\gamma_i$ are elementary cycles (see the proof of Proposition~\ref{propos_cycly}).

5. Then we have
$$
f\bigl(\langle L\rangle\bigr)=\sum_{i=1}^ln_i\langle
c,\gamma_i\rangle,
$$
where the values $\langle c,\gamma_i\rangle$ are taken from Table~\ref{table}.

Now the cycle~$f_{\sharp}(K)$ representing the Poincar\'e dual of the first rational Pontryagin class of an oriented combinatorial manifold~$K$ can be computed by the universal local formula~(\ref{eq_uniform}).

In the algorithm described steps~1 and~3 are the most complicated. It follows from Pachner's theorem that sequence~(\ref{eq_bistellar}) exists for any three-dimensional combinatorial sphere~$L$. This sequence can be found by performing all possible bistellar moves starting from~$L$. However, such algorithm is extremely complicated. Using some empirical rules for finding simplifying bistellar moves Bj\"orner and Lutz created a computer program BISTELLAR (see~\cite{BL00}). In particular, this program allows to find rather effectively sequences~(\ref{eq_bistellar}) for combinatorial spheres with small number of vertices. However, no theoretical estimate for the number of bistellar moves in a sequence~(\ref{eq_bistellar}) for a given combinatorial sphere~$L$ is known.

The complexity of step~$3$ consists in branching of the recurrence. While the complexity of step~$1$ seems to be unavoidable, the complexity of step~$3$ is caused by our wish to obtain the universal local formula~(\ref{eq_uniform}). This step can be essentially simplified if we do not require our formula to be local. Let us describe the simplified procedure for computing a simplicial cycle~$Z$ representing the Poincar\'e dual of the first rational Pontryagin class of an oriented $m$-dimensional combinatorial manifold~$K$.

1. For every simplex~$\sigma\in K$ of codimension either~$3$ or~$4$, we find a sequence of bistellar moves
\begin{equation}\label{eq_bistellar2}
\partial\Delta^{\codim\sigma-1}=L_1^{(\sigma)}
\stackrel{\beta_1^{(\sigma)}}{\rightsquigarrow}L_2^{(\sigma)}
\stackrel{\beta_2^{(\sigma)}}{\rightsquigarrow}\ldots
\stackrel{\beta_{k(\sigma)}^{(\sigma)}}{\rightsquigarrow}L_{k(\sigma)+1}^{(\sigma)}\cong
\Lk\sigma.
\end{equation}

2. For every codimension~$4$ simplex~$\sigma\in K$  we choose an arbitrary orientation of it. We endow every codimension~$3$ simplex~$\tau\in K$ such that $\tau\supset\sigma$ with the orientation so as to obtain that the incidence coefficients of simplices~$\sigma$ and~$\tau$ is equal to~$+1$. (A simplex $\tau$ can have different orientations for different simplices $\sigma$ containing it.) We endow the combinatorial sphere~$L_j^{(\sigma)}$ and the combinatorial spheres~$L_j^{(\tau)}$ with orientations agreed with the chosen orientations of simplices~$\sigma$ and~$\tau$. We put,
$$
\zeta_{\sigma}=\sum_{j=1}^{k(\sigma)}\sum_{v\in
U\left(\beta_j^{(\sigma)}\right)}\left\{\left(\beta_j^{(\sigma)}\right)_v\right\}-
\mathop{\sum_{\tau\in
K,\,\tau\supset\sigma,}}\limits_{\codim\tau=3}
\sum_{j=1}^{k(\tau)}\left\{\beta_j^{(\tau)}\right\}.
$$
Then $\zeta_{\sigma}$ is a cycle.

3. For every $\sigma$, we decompose the cycle $\zeta_{\sigma}$ into a linear combination of elementary cycles and compute the value~$r_{\sigma}=\langle c,\zeta_{\sigma}\rangle$. Then the required cycle is given by
$$
Z=\sum_{\sigma\in K,\,\codim\sigma=4}r_{\sigma}\sigma.
$$

\section{Conclusion}
\label{section_zakluch}

Notice that the role played by bistellar moves in the formula described in sections~\ref{section_locform}--\ref{section_explicit} is very similar to the role played by configuration spaces in the Gabrielov--Gelfand--Losik formula. Actually, a three-dimensional combinatorial sphere is a very abstract object, which is almost impossible to work with without relating it to some simpler object. In the Gabrielov--Gelfand--Losik formula such relation is provided by a linear on simplices embedding of the cone over a three-dimensional combinatorial sphere into~$\R^4$. In the author's formula such relation is provided by a sequence of bistellar move transforming a three-dimensional combinatorial sphere to the boundary of a four-dimensional simplex. An advantage of a sequence of bistellar moves consists in the fact that it exists for any combinatorial sphere. Indeed, both bistellar moves and configuration spaces are organized rather easy for combinatorial spheres of dimensions~$\le 2$ and turn out to be much more complicated for three-dimensional spheres. This is the main reason for complexity of both combinatorial formulae considered in this paper.

In conclusion I wish to thank V.\,M.~Buchstaber who suggested me a problem on combinatorial formulae for the Pontryagin classes several years ago. I am also grateful to G.\,I.~Sharygin for useful discussion.

\end{document}